\documentclass[11pt]{amsart}

\usepackage{amsthm, amsfonts, amssymb, amscd, rotating}
\usepackage[pagebackref,colorlinks]{hyperref}
\usepackage{tikz-cd}
\usepackage{geometry}
\usepackage{marginnote}
\usepackage{aligned-overset}

\theoremstyle{definition}
\newtheorem{ntn}{Notation}[section]

\theoremstyle{plain}
\newtheorem{lem}[ntn]{Lemma}
\newtheorem{prp}[ntn]{Proposition}
\newtheorem{thm}[ntn]{Theorem}

\newtheorem{conj}[ntn]{Conjecture}
\theoremstyle{definition}

\newtheorem{rem}[ntn]{Remark}
\newtheorem{exa}[ntn]{Example}

\numberwithin{equation}{section}

\newcommand{\N}{\mathbb{N}}
\newcommand{\z}{\mathbb{Z}}
\newcommand{\q}{\mathbb{Q}}

\newcommand{\F}{\mathbb{F}}

\newcommand{\GG}{\mathcal{G}}

\newcommand{\EE}{\mathcal{E}}

\newcommand{\KK}{\mathcal{K}}
\newcommand{\LL}{\mathcal{L}}

\newcommand{\ppp}{\mathfrak{p}}
\newcommand{\mmm}{\mathfrak{m}}

\newcommand{\lan}{\langle}
\newcommand{\ran}{\rangle}
\newcommand{\se}{\subseteq}
\newcommand{\arr}{\rightarrow}
\newcommand{\larr}{\longrightarrow}
\newcommand{\harr}{\hookrightarrow}
\newcommand{\two}{\twoheadrightarrow}

\newcommand{\B}{{\rm B}}
\newcommand{\Ee}{{\rm E}}
\newcommand{\GL}{{\rm GL}}
\newcommand{\SL}{{\rm SL}}
\newcommand{\GE}{{\rm GE}}
\newcommand{\im}{{\rm im}}

\newcommand{\ab}{{\rm ab}}

\newcommand{\PSL}{{\rm PSL}}

\newcommand{\rank}{{\rm rank}}

\newcommand{\spcd}{{\rm scpd}}

\newcommand {\mtxx}[4]
{\left(\!
\begin{array}{cc}
\!\!#1 & \!\!#2 \\
\!\!#3 & \!\!#4
\end{array}\!\!
\right)}

\newtheoremstyle{athm}
{}
{}
{\itshape}
{}
{\scshape}
{}
{.5em}
{\thmnote{#3}}
\theoremstyle{athm}
\newtheorem*{athm}{}

\begin{document}

\title{The second integral homology of ${\rm SL}_2(\mathbb{Z}[1/n])$}
\author{Behrooz Mirzaii, Bruno R. Ramos, Thiago Verissimo}
\address{\sf Instituto de Ci\^encias Matem\'aticas e de Computa\c{c}\~ao (ICMC), 
Universidade de S\~ao Paulo, S\~ao Carlos, São Paulo, Brasil}
\email{bmirzaii@icmc.usp.br}
\email{brramos@usp.br}
\email{thiagovlg@usp.br}

\begin{abstract}
In this article, we explore the second integral homology, or Schur multiplier, of the special linear group 
${\rm SL}_2(\mathbb{Z}[1/n])$ for a positive integer $n$. We definitively calculate the group structure of 
$H_2({\rm SL}_2(\mathbb{Z}[1/n]),\mathbb{Z})$ 
when $n$ is divisible by one of the primes $2$, $3$, $5$, $7$ or $13$. For a general $n > 1$, 
we offer a partial description by placing the homology group within an exact sequence, and we investigate 
its rank. Finally, we propose a conjectural structure for $H_2({\rm SL}_2(\mathbb{Z}[1/n]),\mathbb{Z})$ when $n$ is not divisible 
by any of those specific primes.\\



\end{abstract}

\maketitle
%%%%%%%%%%%%%%%%%%%%%%%%%%%%%%%%%%%%%%%%%%%%%%%%%%%%%%%%%%%%%%%%%%%%%%%%%%%%%%%%%%
\section*{Introduction}
%%%%%%%%%%%%%%%%%%%%%%%%%%%%%%%%%%%%%%%%%%%%%%%%%%%%%%%%%%%%%%%%%%%%%%%%%%%%%%%%%%

The (co)homology groups of $\SL_2(\z[1/n])$ are of considerable interest and importance, 
finding applications in diverse fields of mathematics such as number theory, algebraic 
$K$-theory, hyperbolic geometry and the theory of modular and automorphic forms. These 
groups offer valuable insights into the arithmetic properties of the ring $\z[1/n]$ and 
related rings.

Since $\SL_2(\z[1/n])$ is an arithmetic group, its (co)homology groups are known to be finitely 
generated. However, the exact determination of their group structure remains a challenging and 
important problem, which has been the subject of many research articles (see, for example, 
\cite{moss1980}, \cite{an1998}, \cite{ww1998}, \cite{ae2014}, \cite{h2016},  \cite{B-E2024} 
and \cite{carl2024}).

The study of (co)homology groups of $\SL_2(\z[1/n])$ has seen considerable progress. Adem and Naffah 
\cite{an1998} fully computed these groups for $\SL_2(\z[1/p])$, where $p$ is a prime number. Bui and 
Ellis \cite{ae2014} employed computational methods to calculate homology groups for $n \leq 50$ (with 
few exceptions). Hutchinson \cite{h2016} later determined the second homology when $6$ divides $n$. 
More recently, \cite{B-E2024} and \cite{carl2024}, independently and with different methods, have 
achieved complete calculations of the first homology groups for arbitrary $n$ (see Theorem \ref{H1} 
in Section \ref{sec2}).

The primary focus of this article is the investigation of the second integral homology 
of the group $\SL_2(\z[1/n])$. We may always assume that $n$ is square-free. Building on Hutchinson's 
insights, incorporating new ideas and drawing upon the results outlined in the preceding paragraph, 
we present a structural theorem for $H_2(\SL_2(\z[1/n]),\z)$ provided $n$ is 
%\textcolor{red}{(Remark 1) a squarefree integer which is} 
divided by one of the primes $2$, $3$, $5$, $7$ or 
$13$. We state this as the next theorem (see Theorem \ref{rp=1}). 
%Note that $2.3.5.7.13=2730$.

%\textcolor{red}{It would be nice to have a notation for the smallest common prime divisor(scpd)
%... (-, -) stands for gcd. For instance, scpd(26, 2730) = 2, while gcd(26, 2730) = 26. Let us 
%consider the following symbol to denote the scpd of two numbers a and b: scd(a,b).}

\begin{athm}[{\bf Theorem A.}]
Let $n$ be a square-free positive integer and let
%which is divided by one of the primes $2$, $3$, $5$, $7$ or $13$. 
$\spcd(n,2730)$ be the smallest common prime divisor of $n$ and $2730=2\cdot 3 \cdot 5 \cdot 7 \cdot 13$.
%If an odd prime $p$ divides $n$, let $p-1=2^{s_p}m_p$, where $s_p \geq 1$ and $m_p$ is odd.
\par {\rm (i)} If $\spcd(n, 2730)=2$, i.e. $n$ is even, then 
%\textcolor{blue}{(If $(n, 2\cdot3\cdot5\cdot7\cdot13 = 2730)=2$, then (Objection 1...))}
\[
%\begin{array}{c}
H_2(\SL_2(\z[1/n]),\z) \simeq \z\oplus  \bigoplus_{p\mid (n/2)} \z/(p-1).
%\end{array}
\]
\par {\rm (ii)} If $\spcd(n, 2730)=3$, then 
%\textcolor{blue}{(If $(n,2\cdot3\cdot5\cdot7\cdot13 =2730)=3$, then)}
\[
%\begin{array}{c}
H_2(\SL_2(\z[1/n]),\z) \simeq \z\oplus  \bigoplus_{p\mid (n/3)} \z/(p-1).
%\end{array}
\]
\par {\rm (iii)} If $\spcd(n, 2730)=5$, then 
%\textcolor{blue}{(If $(n,2\cdot3\cdot5\cdot7\cdot13 =2730)=5$, then)}
\[
H_2(\SL_2(\z[1/n]),\z) \!\simeq\! \z \oplus
\begin{cases}
\z/2  \oplus \displaystyle\bigoplus_{p\mid (n/5)}\z/(p-1) & \text{if $p\equiv 1\!\!\!\! \pmod {4}$ for all $p\mid (n/5)$},\\
\z/4\oplus \z/((q-1)/2) \! \oplus \!\!\!\!\displaystyle\bigoplus_{p\mid (n/5q)}\!\!\! \z/(p-1) & \text{if $q\equiv 3 \!\!\!\! 
\pmod {4}$ for some $q\mid (n/5)$.}
\end{cases}
\]
\par {\rm (iv)} If $\spcd(n, 2730)=7$, then 
%\textcolor{blue}{(If $(n,2\cdot3\cdot5\cdot7\cdot13 =2730)=7$, then)}
\[
%\begin{array}{c}
H_2(\SL_2(\z[1/n]),\z) \simeq \z\oplus \z/3\oplus \bigoplus_{p\mid (n/7)} \z/(p-1).
%\end{array}
\]
\par {\rm (v)} 
If $\spcd(n, 2730)=13$, then 
\[
H_2(\SL_2(\z[1/n]),\z) \!\simeq \! \z \oplus
\begin{cases}
\z/6  \oplus \displaystyle\bigoplus_{p\mid (n/13)}\z/(p-1) &\!\!\! \text{if $p\equiv 1\!\!\!\! \pmod {4}$ for all $p\mid (n/13)$},\\
\z/12 \oplus \z/((q-1)/2) \! \oplus \!\!\!\!\!\displaystyle\bigoplus_{p \mid (n/13q)}\!\!\!\!\! \z/(p-1) &\!\!\! \text{if $q\equiv 3 \!\!\!\! \pmod {4}$ for some $q\mid (n/13)$.}
\end{cases}
\]
\end{athm}
\bigskip
 
It is a known fact, proved by Adem and Naffah in \cite{an1998}, that the second homology group
$H_2(\SL_2(\z[1/p]),\z)$, $p$ a prime, is of rank one precisely when $p$ is one of the primes $2$, $3$, $5$, 
$7$ or $13$. More precisely,  if
\[
r_p:=\rank\ H_2(\SL_2(\z[1/p]),\z),
\]
then  $r_p=1$ if and only if $p=2$, $3$, $5$, $7$ or $13$ (see Theorem \ref{AN} below).

For any square-free integer $n$, we establish the following theorem (see Theorem \ref{exact2}) that provides 
insights into the structure of $H_2(\SL_2(\z[1/n]),\z)$. Notably, certain parts of {\sf Theorem A} 
are derived directly from this broader theorem.

\begin{athm}[{\bf Theorem B.}]
Let $n$ be a square-free positive integer with prime decomposition $n=p_1\cdots p_l$ and let 
$r_{p_1}\leq \dots \leq r_{p_l}$. If $r_{p_i}=r_{p_{i+1}}$, we assume that $p_i< p_{i+1}$. 
Then we have the exact sequence
\[
\begin{array}{c}
H_2(\SL_2(\z[1/p_1]),\z) \arr H_2(\SL_2(\z[1/n]),\z) \arr \bigoplus_{i=2}^l \z/(p_i-1) \arr 0.
\end{array}
\]
\end{athm}
\bigskip

{\sf Theorem B} follows from a careful analysis of the Mayer-Vietoris exact sequence applied to the 
amalgamated product
\[
\begin{array}{c}
\SL_2(\z[1/pn])\simeq \SL_2(\z[1/n]) \ast_{\Gamma_0(n,p)} \SL_2(\z[1/n]),
\end{array}
\]
where $p$ is a prime not dividing $n$ and $\Gamma_0(n,p)$ is the following subgroup of $\SL_2(\z[1/n])$:
\[
\begin{array}{c}
\Gamma_0(n, p):=\bigg\{{\mtxx a b c d}\in \SL_2(\z[1/n]) : p\mid c\bigg\}.
\end{array}
\]
For more on this isomorphism, we refer the reader to \cite[p. 80]{serre1980} (see also Theorem \ref{iso-amal}
in the next section).

Understanding the first and the second homology of $\Gamma_0(n, p)$ is essential for proving {\sf Theorem B}.
Specifically, the structure of the first homology, $H_1(\Gamma_0(n,p),\z)$, is determined and detailed 
as follows (see Theorem \ref{G0}).

\begin{athm}[{\bf Theorem C.}]
Let $n$ be a natural number greater than one, $p$ a prime not dividing $n$ and $d:= \gcd\{m^2-1: m\mid n\}$. 
%Moreover, let $\F_p$ be the finite field with $p$ elements.
\par {\rm (i)} If $p>3$ and $p\nmid d$, then
\[
\begin{array}{c}
H_1(\Gamma_0(n,p),\z)\simeq  H_1(\SL_2(\z[1/n]),\z) \oplus \z/(p-1).
\end{array}
\]
\par {\rm (ii)} If $p=3$ and $d=3t$, where $3\nmid t$ (e.g. when $2\mid n$ or $5\mid n$), then 
\[
\begin{array}{c}
H_1(\Gamma_0(n, 3),\z)\simeq  H_1(\SL_2(\z[1/n]),\z) \oplus \z/2 \oplus \z/3.
\end{array}
\]
\par {\rm (iii)} If $p=2$ and $d=8t$, where $2\nmid t$ (e.g. when $3\mid n$ or $5\mid n$), then 
\[
\begin{array}{c}
H_1(\Gamma_0(n,2),\z)\simeq  H_1(\SL_2(\z[1/n]),\z) \oplus \z/2 \oplus \z/4.
\end{array}
\]
\end{athm}
\bigskip

We conclude the introduction by outlining the structure of the present article. In Section~\ref{sec1}, 
we review established results concerning the group $\SL_2(\z[1/n])$, specifically its congruence 
subgroup property and its description as an amalgamated product. Section~\ref{sec2} reviews necessary 
results on the first and second homology of $\SL_2(\z[1/n])$. Section~\ref{sec3} provides a brief overview 
of the Mayer-Vietoris exact sequence associated with the amalgamated product decomposition of $\SL_2(\z[1/pn])$, 
where $p$ is a prime not dividing $n$. Section~\ref{sec4} examines the first homology of $\Gamma_0(n,p)$ and 
presents the proof of {\sf Theorem C}. In Section \ref{sec5}, we demonstrate that the inclusion 
$\Gamma_0(n,p) \harr \SL_2(\z[1/n])$ induces a surjective map on second homology. Section~\ref{sec6} 
then provides the proofs of {\sf Theorems B} and {\sf A}. Finally, Section~\ref{sec7} investigates 
the rank of the second homology of $\SL_2(\z[1/n])$ and concludes the article by proposing 
a conjectural structure for the second homology of $\SL_2(\z[1/n])$ when $n$ is not divisible 
by any of the primes $2$, $3$, $5$, $7$ or $13$.
\\
~\\
{\bf Acknowledgments.} We sincerely thank the anonymous referees for their careful reading of the 
manuscript and for providing numerous valuable comments. We are especially grateful for their 
identification of an error in the original statement of the main theorem. 
The second and third authors' contributions to this article were made possible 
by CAPES (Coordena\c{c}\~ao de Aperfeiçoamento de Pessoal de N\'ivel Superior) Ph.D. and M.Sc. fellowships 
(grant numbers 88887.983475/2024-00 and 88887.955905/2024-00).

%%%%%%%%%%%%%%%%%%%%%%%%%%%%%%%%%%%%%%%%%%%%%%%%%%%%%%%%%%%%%%%%%%%%%%%%%%%%%%%%%%%%%%%%%%%%%
\section{The special linear group of degree two over \texorpdfstring{$\z[1/n]$}{Lg}}\label{sec1}
%%%%%%%%%%%%%%%%%%%%%%%%%%%%%%%%%%%%%%%%%%%%%%%%%%%%%%%%%%%%%%%%%%%%%%%%%%%%%%%%%%%%%%%%%%%%%

For any nonzero integer $n$, $\z[1/n]$ is the following subring of $\q$:
\[
\begin{array}{c}
\z[1/n]:=\{a/n^r:a\in \z, r\in \z^{\geq 0} \}.
\end{array}
\]
This is a Euclidean domain. If $m\mid n$, then $\z[1/m]\se \z[1/n]$. 
Moreover, for any $k \neq 0$ and $r \geq 1$, $\z[\pm 1/{k^rm}]=\z[{1}/{km}]$.
Thus, if $n=\pm p_1^{m_1}\cdots p_s^{m_s}$ is the prime factorization of $n$, then 
\[
\begin{array}{c}
\z[1/n]=\z[{1}/{p_1\cdots p_s}].
\end{array}
\]
Therefore, to study $\z[1/n]$, we are allowed to always assume that $n$ is a square-free positive integer.
%\textcolor{red}{(Remarks 1 and 3) This sentence provides an explanation on $n$ being a square-free integer. 
%I believe that this address Remarks 1 and 3.}

For a commutative ring $A$, let $\Ee_2(A)$ be the subgroup of $\GL_2(A)$ generated by the elementary matrices
\[
E_{12}(a):=\begin{pmatrix}
1 & a\\
0 & 1
\end{pmatrix}, \ \ \ \
E_{21}(a):=\begin{pmatrix}
1 & 0\\
a & 1
\end{pmatrix}, \ \ \ \ a\in A.
\]
Clearly, $\Ee_2(A)\se \SL_2(A)$. For examples of rings for which this is a proper subgroup, see 
\cite[Theorem 6.1]{cohn1966}.

We say that $A$ is a $\GE_2$-{\it ring} if $\Ee_2(A)=\SL_2(A)$.
It is known that semilocal rings and Euclidean domains are $\GE_2$-rings (see \cite[p.~245]{s1982} 
and \cite[\S2]{cohn1966}). Hence, $\z[1/n]$ is a $\GE_2$-ring. Indeed, it can be shown that for
any $n>1$, $\SL_2(\z[1/n])$ is generated by the matrices $E_{21}(1)$ and $E_{12}(-1/n)$
(see \cite[p.~204]{men1967} and \cite[Lemma~1.1]{carl2024}).

The next theorem is a special 
case of a general result due to Vaserstein and Liehl (see \cite{vas1972} and \cite{liehl1981}).

\begin{thm}[Vaserstein, Liehl]\label{VL}
Let $n>1$ be an integer and let $I_1$ and $I_2$ be nonzero ideals of $\z[1/n]$. Let
\[
\widetilde{\Gamma}(I_1, I_2) :=\bigg\{ {\mtxx a b c d} \in 
\begin{array}{c}
\SL_2(\z[1/n]) 
\end{array}
:  b \in I_1, c \in I_2, a-1, d-1 \in I_1I_2\bigg\}.
\]
Then $\widetilde{\Gamma}(I_1, I_2)$ is generated by  elementary matrices $E_{12}(x)$, $x \in I_1$, 
and $E_{21}(y)$, $y \in I_2$, and it is of finite index in $\SL_2(\z[1/n])$.
\end{thm}
\begin{proof}
See \cite[Theorem, p. 321]{vas1972} and \cite[\S4]{liehl1981}. 
\end{proof}

\begin{lem}\label{surj-GE2}
Let $I$ be an ideal of a Euclidean domain $A$ and $\pi:A\arr A/I$ be the quotient map of rings.
Then the natural map $\pi_\ast:\SL_2(A) \arr \SL_2(A/I)$ is surjective.
\end{lem}
\begin{proof}
If $I=(0)$, there is nothing to prove. Thus, let $I$ be a nontrivial ideal of $A$. Then $A/I$ is semi-local 
and thus it is a $\GE_2$-ring. It follows from this that $\SL_2(A/I)$ is generated by elementary matrices 
$E_{12}(\overline{a})$ and $E_{21}(\overline{a})$, $a\in A$. Since $\pi_\ast(E_{ij}(a))=E_{ij}(\overline{a})$, $\pi_\ast$ 
is surjective.
\end{proof}

Let $I$ be an ideal of a Euclidean domain $A$. Let $\Gamma(A,I)$ be the kernel of the surjective map
$\pi_\ast:\SL_2(A) \arr \SL_2(A/I)$. Hence,
\[
\Gamma(A, I) :=\bigg\{ {\mtxx a b c d} \in 
\begin{array}{c}
\SL_2(A) 
\end{array}
:  b, c, a-1, d-1 \in I\bigg\}.
\]
Subgroups of the form $\Gamma(A,I)$, for some nontrivial ideal $I$, are called {\it principal congruence 
subgroups} of $\SL_2(A)$. A subgroup of $\SL_2(A)$ is called a {\it congruence subgroup} if it
contains a principal congruence subgroup.

\begin{lem}\label{coprime}
Let $I$ and $J$ be two coprime ideals of a Euclidean domain $A$. Then the composite
$\Gamma(A, I) \arr \SL_2(A) \arr \SL_2(A/J)$ is surjective.
\end{lem}
\begin{proof}
See \cite[Lemma 5.12]{h2016}.
\end{proof}

Let $A$ be a Euclidean domain such that for any nontrivial ideal $I$, $A/I$ is a finite ring. Then 
any principal congruence subgroup of $\SL_2(A)$ is of finite index. We say that $\SL_2(A)$ has the 
{\it congruence subgroup property} if any finite-index subgroup of $\SL_2(A)$ is a congruence subgroup.

The next theorem is a special case of a general result due to Serre (see \cite{serre1970} and \cite{men1967}) 
and shows that, for any $n>1$, $\SL_2(\z[1/n])$ has the congruence subgroup property.

\begin{thm}[Congruence subgroup property]\label{cong}
Let $\Gamma$ be a non-central normal subgroup of $\SL_2(\z[1/n])$, where $n > 1$. Then $\Gamma$ 
contains a subgroup of the form $\Gamma(\z[1/n], I)$, for some nontrivial ideal $I$ of $\z[1/n]$. 
In particular, $\Gamma$ is of finite index in $\SL_2(\z[1/n])$.
\end{thm}
\begin{proof}
This is a special case of \cite[Proposition 2]{serre1970}. 
\end{proof}

Let $n$ be a natural number and $p$ a prime such that $p\nmid n$. Let
\[
\begin{array}{c}
\SL_2(\z[1/n])^{\pi_p}:=\pi_p^{-1} \SL_2(\z[1/n]) \pi_p
=\bigg\{{\mtxx a {p^{-1}b} {pc} d}\in \SL_2(\q):{\mtxx a b c d}\in 
\SL_2(\z[1/n]) \bigg\},
\end{array}
\]
where $\pi_p:={\mtxx p 0 0 1} \in \GL_2(\q)$. It is easy to see that
\[
\begin{array}{c}
\SL_2(\z[1/n])\cap \SL_2(\z[1/n])^{\pi_p}
=\bigg\{{\mtxx a b c d}\in \SL_2(\z[1/n]) : p\mid c\bigg\}.
\end{array}
\]
In the introduction we denoted this group by $\Gamma_0(n, p)$. Now, consider the following inclusions:
\begin{align*}
& \begin{array}{c} i_1:\Gamma_0(n,p) \harr \SL_2(\z[1/n]),\end{array} & &
\begin{array}{c} i_2: \Gamma_0(n,p) \harr \SL_2(\z[1/n])^{\pi_p},\end{array} \\
& \begin{array}{c} j_1:\SL_2(\z[1/n]) \harr \SL_2(\z[1/pn]), \end{array}&&
\begin{array}{c} j_2: \SL_2(\z[1/n])^{\pi_p} \harr \SL_2(\z[1/pn]). \end{array}
\end{align*}

As established in \cite{serre1980}, the Theory of Trees provides a proof for the following well-known result.

\begin{thm}\label{iso-amal}
Let $n$ be a natural number and $p$ a prime such that $p\nmid n$. Then
\[
\begin{array}{c}
\SL_2(\z[1/pn])\simeq \SL_2(\z[1/n]) \ast_{\Gamma_0(n,p)} \SL_2(\z[1/n])^{\pi_p}.
\end{array}
\]
\end{thm}
\begin{proof}
See \cite[p. 80]{serre1980}.
\end{proof}

If we replace $i_2$ and $j_2$ by the injective maps
\[
\begin{array}{c}
i_2': \Gamma_0(n,p) \arr \SL_2(\z[1/n]) 
\ \ \ \ \text{ and }
\ \ \ \
j_2': \SL_2(\z[1/n]) \arr \SL_2(\z[1/pn]),
%{\mtxx a b c d}\mapsto {\mtxx a {pb} {p^{-1}c} d},
\end{array}
\]
which are given by 
\[
\begin{array}{c}
{\mtxx a b c d}\mapsto {\mtxx a {pb} {p^{-1}c} d} \ \ \textrm{and} \ \ {\mtxx a b c d}\mapsto {\mtxx a {p^{-1}b} {pc} d},
\end{array}
\]
respectively, then the above isomorphism finds the following form
\begin{equation*}
\begin{array}{c}
\SL_2(\z[1/pn])\simeq \SL_2(\z[1/n]) \ast_{\Gamma_0(n,p)} \SL_2(\z[1/n]).
\end{array}
\end{equation*}

For a prime $p$, let $\F_p:=\z/p$ be the prime field with $p$ elements. If $p\nmid n$, then the 
natural map $\z[1/n] \arr \F_p$, $a/n^r\mapsto \overline{a}/\overline{n}^{r}$, induces 
the natural surjective homomorphisms 
\[
\begin{array}{c}
\SL_2(\z[1/n]) \two \SL_2(\F_p), \ \ \ \ \ \ \Gamma_0(n,p) \two \B(\F_p),
\end{array}
\]
where
\[
\B(\F_p):=\bigg\{ {\mtxx a b 0 {a^{-1}}}: a\in \F_p^\times, b\in \F_p\bigg\} \se \SL_2(\F_p).
\]
Thus, we have the morphism of extensions
\begin{equation}\label{ext}
\begin{tikzcd}
1 \ar[r] & \Gamma(n,p) \ar[r]\ar[d, equal]& \Gamma_0(n,p) \ar[r]\ar[d, "i_1"]& \B(\F_p) 
\ar[r] \ar[d, hook]& 1\\
1 \ar[r] & \Gamma(n,p) \ar[r]& \SL_2(\z[1/n]) \ar[r]& \SL_2(\F_p) \ar[r]& 1,
\end{tikzcd}
\end{equation}
where 
\[
\begin{array}{c}
\Gamma(n,p):=\bigg\{{\mtxx a b c d}\in \SL_2(\z[1/n]) : p\mid b,c, a-1,d-1\bigg\}
=\Gamma(\z[1/n], \lan p\ran). 
\end{array}
\]
Observe that 
\[
\begin{array}{c}
[\SL_2(\z[1/n]): \Gamma_0(n,p)]=p+1, \ \ \ \ \ [\Gamma_0(n,p): \Gamma(n,p)]=p(p-1),
\end{array}
\]
and thus
\begin{equation}\label{eq0}
\begin{array}{c}
[\SL_2(\z[1/n]): \Gamma(n,p)]=p(p^2 - 1).
\end{array}
\end{equation}
Let 
\[
\begin{array}{c}
\Gamma(n, p^k):=\Gamma(\z[1/n], \lan p^k\ran).
\end{array}
\]
\begin{lem}\label{p3k}
If $p$ is a prime and $n$ is not divisible by $p$, then
\[
[\Gamma(n,p):\Gamma(n, p^k)]=p^{3(k-1)}. 
\]
\end{lem}
%\textcolor{blue}{The exponent is 3(k-1)(Objection Y7 and X2)}
\begin{proof}
See \cite[Corollary 5.11]{h2016}.
\end{proof}

%%%%%%%%%%%%%%%%%%%%%%%%%%%%%%%%%%%%%%%%%%%%%%%%%%%%%%%%%%%%%%%%%%%%%%%%%%%%%%%%%%%%%%%%%%%%%%%%%%%%%%%%%%%%%%%%%
\section{The first and second homology groups of \texorpdfstring{$\SL_2(\z[1/n])$}{Lg}}\label{sec2}
%%%%%%%%%%%%%%%%%%%%%%%%%%%%%%%%%%%%%%%%%%%%%%%%%%%%%%%%%%%%%%%%%%%%%%%%%%%%%%%%%%%%%%%%%%%%%%%%%%%%%%%%%%%%%%%%%

Understanding the exact structure of the finitely generated groups $H_k(\SL_2(\z[1/n]),\z)$ is essential 
for various applications. While $H_0(\SL_2(\z[1/n]),\z)$ is simply $\z$, the higher homology groups are 
more complex. The structural theorem for $H_1(\SL_2(\z[1/n]),\z)$, proved independently and with different 
methods in \cite{B-E2024} and \cite{carl2024}, will be a fundamental tool in this work.

\begin{thm}
%[Mirzaii-Torrez Pérez/Nyberg-Brodda]
\label{H1}
If $n>1$ is an integer, then 
\[
\begin{array}{c}
H_1(\SL_2(\z[1/n]),\z) \simeq 
%\SL_2(\z[1/n])^\ab\simeq 
\end{array}
\begin{cases}
0   & \text{if $2\mid n$, $3\mid n$}  \\
\z/3   & \text{if $2\mid n$, $3\nmid n$}\\
\z/4   & \text{if $2\nmid n$, $3\mid n$}\\
\z/12 &  \text{if $2\nmid  n$, $3\nmid n$}
\end{cases},
\]
which is induced by the map  $\z[1/n] \arr \SL_2(\z[1/n])^\ab$, 
$a \mapsto \overline{E_{12}(a)}$.
\end{thm}
\begin{proof}
See \cite[Proposition 4.4]{B-E2024} or \cite[Theorem 1.2]{carl2024}.
\end{proof}

The proof of the above theorem in \cite{B-E2024} relies partly on the following result, which will 
also be essential in some of the proofs presented in this article.

\begin{prp}\label{exa-2}
Let $A$ be a commutative local ring with maximal ideal $\mmm_A$. Then
\[
H_1(\SL_2(A),\z) \simeq 
%\SL_2(A)^\ab \simeq 
\begin{cases}
A/\mmm_A^2 &  \text{if $|A/\mmm_A|=2$}  \\
A/\mmm_A &  \text{if $|A/\mmm_A|=3$}  \\
0 &  \text{if $|A/\mmm_A|\geq  4$}  \\
\end{cases}.
\]
\end{prp}
\begin{proof}
See \cite[Proposition 4.1]{B-E2024}.
\end{proof}

This paper focuses on determining the group structure of $H_2(\SL_2(\z[1/n]),\z)$. Our proofs 
rely on two key results concerning this group, which are already established in the literature.

\begin{thm}[Adem-Naffah]\label{AN}
If $p$ is a prime number, then
\[
\begin{array}{c}
H_2(\SL_2(\z[1/p]), \z)\simeq \begin{cases}
\z & \text{if $p=2,3$}\\
\z^{(p-7)/6}\oplus \z/6 & \text{if $p\equiv 1 \pmod {12}$}\\
\z^{(p+1)/6}\oplus\z/2 & \text{if $p\equiv 5 \pmod {12}$}\\
\z^{(p-1)/6}\oplus\z/3 & \text{if $p\equiv 7 \pmod {12}$}\\
\z^{(p+7)/6} & \text{if $p\equiv 11 \!\!\!\pmod {12}$}\\
\end{cases}.
\end{array}
\]
\end{thm}
\begin{proof}
Adem and Naffah gave a complete description of the structure of the cohomology groups $H^k(\SL_2(\z[1/p]),\z)$, 
for any $k\geq 0$ (see \cite[pp. 7-9]{an1998}). Using this and \cite[Lemma~4.3]{B-E-2025}, one can calculate the
homology groups $H_k(\SL_2(\z[1/p]),\z)$. In particular, we obtain the group structure of 
$H_2(\SL_2(\z[1/p]), \z)$, as claimed in this theorem.
\end{proof}

\begin{thm}[Hutchinson]\label{H26}
Let $n$ be a square-free integer such that $6\mid n$. Then 
\[
\begin{array}{c}
H_2(\SL_2(\z[1/n]),\z)\simeq \z\oplus \bigoplus_{p\mid n} \z/(p-1).
\end{array}
\]
More generally, if $m\mid n$ and $6\mid m$, then we have the split exact sequence
\[
\begin{array}{c}
0 \arr H_2(\SL_2(\z[1/m]),\z) \arr H_2(\SL_2(\z[1/n]),\z) \arr 
\bigoplus_{ p\mid (n/m)} \F_p^\times \arr 1.
\end{array}
\]
\end{thm}
\begin{proof}
See \cite[Theorem 6.10 and Theorem 6.12]{h2016}.  
\end{proof}

For the map $H_2(\SL_2(\z[1/n]),\z) \arr \bigoplus_{p\mid (n/m)} \F_p^\times$, appearing in 
Theorem \ref{H26}, we refer the reader to Section~\ref{sec6}.
 
This article presents a generalization of the results of Hutchinson, Adem-Naffah, and Bui-Ellis, 
determining the group structure of $H_2(\SL_2(\z[1/n]),\z)$ for any $n\in \N$ that has a 
prime factor within the set $\{2, 3, 5, 7, 13\}$.

%%%%%%%%%%%%%%%%%%%%%%%%%%%%%%%%%%%%%%%%%%%%%%%%%%%%%%%%%%%%%%%%%%%%%%%%%%%%%%%%%%
\section{Mayer-Vietoris exact sequence}\label{sec3}
%%%%%%%%%%%%%%%%%%%%%%%%%%%%%%%%%%%%%%%%%%%%%%%%%%%%%%%%%%%%%%%%%%%%%%%%%%%%%%%%%%

Let $n$ be an integer and let $p$ be a prime such that $p\nmid n$.
By the Mayer-Vietoris exact sequence \cite[Corollary 7.7, \S7, Chap. II]{brown1994} applied to the amalgamated 
isomorphism of Theorem \ref{iso-amal}, we have the exact sequence
\begin{align*}
\begin{array}{c}
H_2(\Gamma_0(n,p),\z) \overset{\alpha_2}{\larr} H_2(\SL_2(\z[1/n]),\z) 
\oplus H_2(\SL_2(\z[1/n]),\z) 
\overset{\beta_2}{\larr}  H_2(\SL_2(\z[1/pn]),\z)\\
\\
\larr H_1(\Gamma_0(n,p),\z) \overset{\alpha_1}{\larr} H_1(\SL_2(\z[1/n]),\z)
\oplus H_1(\SL_2(\z[1/n]),\z)\\
\\
\overset{\beta_1}{\larr} H_1(\SL_2(\z[1/pn]),\z) \arr 0.
\end{array}
\end{align*}
Here, $\alpha_k(x)=({i_1}_\ast(x), {i_2'}_\ast(x))$ and $\beta_k(y, z)={j_2'}_\ast(z)-{j_1}_\ast(y)$. 
Based on this and Theorem~\ref{H1}, $\beta_1$ is given by 
\[
(\overline{a}, \overline{b}) \mapsto \overline{pb-a}.
\]

Moreover, the Lyndon/Hochschild-Serre spectral sequence of the morphism of extensions~(\ref{ext}) 
gives us the commutative diagram with exact rows
\[
\begin{tikzcd} 
H_2(\B(\F_p),\!\z)\! \ar[r]\ar[d]&\!\! H_1(\Gamma(n,p),\z)_{\B(\F_p)} \!\ar[r]\ar[d]& 
\!\! H_1(\Gamma_0(n,p),\z) \!\ar[r, two heads] \ar[d]& \!\! H_1(\B(\F_p),\!\z) \ar[d]\\
H_2(\SL_2(\F_p),\!\z) \!\ar[r]& \!\! H_1(\Gamma(n,p),\z)_{\SL_2(\F_p)} \!\ar[r]& 
\!\! H_1(\SL_2(\z[1/n]),\z) \!\ar[r, two heads]& \!\! H_1(\SL_2(\F_p),\!\z)
\end{tikzcd}
\]
(see \cite[Corollary 6.4, Chap. VII]{brown1994}). It is known that $H_2(\SL_2(\F_p),\z)=0$ 
\cite[Theorem 3.9]{beyl1986}. Moreover, by Proposition~\ref{exa-2}, we have
\[
H_1(\SL_2(\F_p),\z)\simeq (\F_p)_{\F_p^\times}=
\begin{cases}
\F_2 &  \text{if $p=2$}  \\
\F_3 &  \text{if $p=3$}  \\
0 &  \text{if $p > 3$}  \\
\end{cases},
\]
where $\F_p^\times$ acts on $\F_p$ by the formula $a.x:=a^2x$.
The homology groups of $\B(\F_p)$ can be obtained by studying the Lyndon/Hochschild-Serre 
spectral sequence of the split extension
\[
0 \arr \F_p \overset{i}{\arr} \B(\F_p) \overset{pr}{\arr} \F_p^\times \arr 1,
\]
where $i(x)={\mtxx 1 x 0 1}$, ${pr}\Big({\mtxx a b 0 {a^{-1}}}\Big)=a$ and a section
%\textcolor{red}{Professor, please, just a question of notation: splitting map is the same as section (right inverse)?} 
%\textcolor{blue}{Yes. But I thing SPLITTING is used usually in Module Theory! Maybe SECTION is more suitable for groups.
%Better to check this} \textcolor{red}{OK! I'm checking David's dissertation, the usual terminology for groups is section!}
$s: \F_p^\times \arr \B(\F_p)$ can be given by  $a \mapsto {\mtxx a 0 0 {a^{-1}}}$. From this, 
we obtain the isomorphisms $H_1(\B(\F_p),\z) \simeq \F_p^\times \oplus (\F_p)_{\F_p^\times}$ and
$H_2(\B(\F_p),\z)\simeq H_1(\F_p^\times, \F_p)$. Hence,
\[
H_1(\B(\F_p),\z)\simeq
\begin{cases}
\F_2 &  \text{if $p=2$}  \\
\F_3^\times \oplus \F_3 &  \text{if $p=3$}  \\
\F_p^\times &  \text{if $p > 3$}  \\
\end{cases},
\ \ \ \ \ \ \ H_2(\B(\F_p),\z)=0
\]
(for the calculation of the second homology we used the calculation of the homology of finite cyclic groups 
given in \cite[pp. 58-59]{brown1994}). Thus, we have the commutative diagram with exact rows
\begin{equation}\label{diagram}
\begin{tikzcd} 
0\!\! \ar[r]&\!\! H_1(\Gamma(n,p),\!\z)_{\B(\F_p)}\!\! \ar[r]\ar[d]& 
\!\! H_1(\Gamma_0(n,p),\!\z) \!\! \ar[r] \ar[d]& \!\! H_1(\B(\F_p),\!\z)\ar[r] \ar[d]&\!\! 0\\
0\!\! \ar[r]& \!\! H_1(\Gamma(n,p),\!\z)_{\SL_2(\F_p)} \!\!\ar[r]& 
\!\! H_1(\SL_2(\z[1/n]),\!\z) \!\! \ar[r]& \!\! H_1(\SL_2(\F_p),\!\z)\ar[r] & \!\! 0.
\end{tikzcd}
\end{equation}

We break the study of the Mayer-Vietoris exact sequence and the above diagram into three cases.\\
\bigskip
{\bf Case (i) $p>3$.} Since $p\nmid n$, by Theorem \ref{H1},  
$H_1(\SL_2(\z[1/n]),\z) \!\simeq \! H_1(\SL_2(\z[1/pn]),\z)$.
Thus, we have the Mayer-Vietoris exact sequence
\begin{align*}
%\label{MV1}
\begin{array}{c}
H_2(\Gamma_0(n, p),\z) \arr H_2(\SL_2(\z[1/n]),\z) \oplus H_2(\SL_2(\z[1/n]),\z) 
\arr  H_2(\SL_2(\z[1/pn]),\z)\\
\\
\arr H_1(\Gamma_0(n, p),\z) \arr H_1(\SL_2(\z[1/n]),\z) \arr 0.
\end{array}
\end{align*}
Since $p>3$, $H_1(\B(\F_p),\z)\simeq\F_p^\times$ and $H_1(\SL_2(\F_p),\z)=0$.
Hence, the diagram (\ref{diagram}) finds the following form:
\begin{equation}\label{diag1}
\begin{tikzcd} 
0 \ar[r]& H_1(\Gamma(n, p),\z)_{\B(\F_p)} \ar[r]\ar[d, two heads]& 
H_1(\Gamma_0(n, p),\z) \ar[r] \ar[d, two heads, "{i_1}_\ast"]& \F_p^\times\ar[r] & 1\\
& H_1(\Gamma(n, p),\z)_{\SL_2(\F_p)} \ar[r, "\simeq"]& 
H_1(\SL_2(\z[1/n]),\z). &   & 
\end{tikzcd}
\end{equation}
{\bf Case (ii)} $p=3$. This case has two parts:
\par {\bf (ii-a)} $2\mid n$. In this case, by Theorem \ref{H1}, we have 
$H_1(\SL_2(\z[1/n]),\z)\simeq \z/3$ and $H_1(\SL_2(\z[{1}/{3n}]),\z)=0$. Thus,
we have the Mayer-Vietoris exact sequence
\begin{align*}
\begin{array}{c}
H_2(\Gamma_0(n,3),\z) \arr H_2(\SL_2(\z[1/n]),\z) \oplus H_2(\SL_2(\z[1/n]),\z) 
\arr  H_2(\SL_2(\z[{1}/{3n}]),\z)\\
\\
\arr H_1(\Gamma_0(n,3),\z) \arr \z/3 \oplus \z/3 \arr 0.
\end{array}
\end{align*}
Since $H_1(\B(\F_3),\z)\simeq\F_3^\times \oplus \F_3$, $H_1(\SL_2(\F_3),\z)\simeq \F_3$ and 
the map 
\[
H_1(\SL_2(\z[1/n]),\z) \arr H_1(\SL_2(\F_3),\z)
\]
is an isomorphism, we have 
$H_1(\Gamma(n,3),\z)_{\SL_2(\F_3)}=0$. Hence, the diagram (\ref{diagram}) becomes
\begin{equation}\label{diag--1}
\begin{tikzcd} 
0 \ar[r]& H_1(\Gamma(n,3),\z)_{\B(\F_3)} \ar[r]& 
H_1(\Gamma_0(n,3),\z) \ar[r] \ar[d, two heads, "{i_1}_\ast"]& \F_3^\times \oplus \F_3\ar[r]\ar[d] & 0\\
&  & 
H_1(\SL_2(\z[1/n]),\z)\ar[r, "\simeq"] & \F_3.  & 
\end{tikzcd}
\end{equation}
\par {\bf (ii-b)} $2\nmid n$. In this case, $H_1(\SL_2(\z[1/n]),\z) \simeq  \z/12$,  
$H_1(\SL_2(\z[1/3n]),\z)\simeq\z/4$ and, thus, we have the Mayer-Vietoris exact sequence
\begin{align*}
\begin{array}{c}
H_2(\Gamma_0(n,3),\z) \arr H_2(\SL_2(\z[1/n]),\z) \oplus H_2(\SL_2(\z[1/n]),\z) 
\arr  H_2(\SL_2(\z[1/3n]),\z)\\
\\
\arr H_1(\Gamma_0(n,3),\z) \arr \z/12 \oplus \z/12 \arr \z/4 \arr 0.
\end{array}
\end{align*}
Since $H_1(\SL_2(\z[1/n]),\z) \arr H_1(\SL_2(\F_3),\z)\simeq\z/3$ is surjective, we have  the isomorphism
$H_1(\Gamma(n,3),\z)_{\SL_2(\F_3)}\simeq \z/4$.
Hence, the diagram (\ref{diagram}) finds the following form:
\begin{equation}\label{diag-1}
\begin{tikzcd} 
0 \ar[r]& H_1(\Gamma(n,3),\z)_{\B(\F_3)} \ar[r]\ar[d, two heads]& 
H_1(\Gamma_0(n,3),\z) \ar[r] \ar[d]& \F_3^\times \oplus \F_3\ar[r]\ar[d] & 0\\
0 \ar[r]& \z/4 \ar[r]& 
H_1(\SL_2(\z[1/n]),\z)\ar[r] & \F_3 \ar[r]  & 0.
\end{tikzcd}
\end{equation}
{\bf Case (iii)} $p=2$. This case has two parts.
\par {\bf (iii-a)} $3\mid n$. In this case, $H_1(\SL_2(\z[1/n]),\z) \simeq \z/4$,  
$H_1(\SL_2(\z[1/2n]),\z)=0$ and, thus, we have the Mayer-Vietoris exact sequence
\begin{align*}
\begin{array}{c}
H_2(\Gamma_0(n,2),\z) \arr H_2(\SL_2(\z[1/n]),\z) \oplus H_2(\SL_2(\z[1/n]),\z) 
\arr  H_2(\SL_2(\z[1/2n]),\z)\\
\\
\arr H_1(\Gamma_0(n,2),\z) \arr \z/4 \oplus \z/4 \arr 0.
\end{array}
\end{align*}
Since $H_1(\B(\F_2),\z)\simeq\F_2$, $H_1(\SL_2(\F_2),\z)\simeq\F_2$ and 
$H_1(\SL_2(\z[1/n]),\z) \arr H_1(\SL_2(\F_2),\z)$ is surjective, we have  
$H_1(\Gamma(n,2),\z)_{\SL_2(\F_2)}\simeq\z/2$.
Hence, the diagram (\ref{diagram}) finds the following form:
\begin{equation}\label{diag--11}
\begin{tikzcd} 
0 \ar[r]& H_1(\Gamma(n,2),\z)_{\B(\F_2)} \ar[r]\ar[d]& 
H_1(\Gamma_0(n,2),\z) \ar[r] \ar[d]&\F_2\ar[r]\ar[d] & 0\\
0 \ar[r] & \z/2  \ar[r]& 
\z/4\ar[r] & \F_2 \ar[r] & 0.
\end{tikzcd}
\end{equation}
\par {\bf (iii-b)} $3\nmid n$. In this case,
$H_1(\SL_2(\z[1/n]),\z) \simeq \z/12$, $H_1(\SL_2(\z[1/2n]),\z)\simeq\z/3$
and, thus, we have the Mayer-Vietoris exact sequence
\begin{align*}
\begin{array}{c}
H_2(\Gamma_0(n,2),\z) \arr H_2(\SL_2(\z[1/n]),\z) \oplus H_2(\SL_2(\z[1/n]),\z) 
\arr  H_2(\SL_2(\z[1/2n]),\z)\\
\\
\arr H_1(\Gamma_0(n,2),\z) \arr \z/12 \oplus \z/12 \arr \z/3 \arr 0.
\end{array}
\end{align*}
Since $H_1(\SL_2(\z[1/n]),\z) \arr H_1(\SL_2(\F_2),\z)$ is surjective, we have  
$H_1(\Gamma(n,2),\z)_{\SL_2(\F_2)}\simeq \z/6$. Hence, the diagram (\ref{diagram}) is of the following form:
\begin{equation}\label{diag-11}
\begin{tikzcd} 
0 \ar[r]& H_1(\Gamma(n,2),\z)_{\B(\F_2)} \ar[r]\ar[d, two heads]& 
H_1(\Gamma_0(n,2),\z) \ar[r] \ar[d]& \F_2\ar[r]\ar[d] & 0\\
0 \ar[r]& \z/6 \ar[r]& \z/12\ar[r] & \F_2 \ar[r]  & 0.
\end{tikzcd}
\end{equation}

%%%%%%%%%%%%%%%%%%%%%%%%%%%%%%%%%%%%%%%%%%%%%%%%%%%%%%%%%%%%%%%%%%%%%%%%%%%%%%%%%%
\section{The first homology of \texorpdfstring{$\Gamma_0(n,p)$}{Lg}}\label{sec4}
%%%%%%%%%%%%%%%%%%%%%%%%%%%%%%%%%%%%%%%%%%%%%%%%%%%%%%%%%%%%%%%%%%%%%%%%%%%%%%%%%%

We denote the natural inclusion $\Gamma(n,p)\arr \SL_2(\z[1/n])$ by $i$.
\begin{thm}\label{p-group}
Let $n>1$ be a square-free integer and $p$ a prime such that $p\nmid n$. Then the kernel of the natural  map 
\[
\begin{array}{c}
i_\ast: H_1(\Gamma(n,p),\z) \arr H_1(\SL_2(\z[1/n]),\z)
\end{array}
\]
is a $p$-group.
\end{thm}
\begin{proof}
Let $A:=\z[1/n]$, $\ppp := \lan p\ran$, $\KK:=\ker(i_\ast)$ and $\GG:=\im(i_\ast)$. 
Observe that $\Gamma(A, \ppp)=\Gamma(n,p)$. By these notations, we have
\[
|\KK|=|H_1(\Gamma(A, \ppp),\z)|/|\GG|.
\]

Since  $[\Gamma(A, \ppp),\Gamma(A, \ppp)]$ is a noncentral normal subgroup of $\SL_2(A)$, by 
Theorem \ref{cong}, the group $[\Gamma(A, \ppp),\Gamma(A, \ppp)]$ contains a subgroup of the 
form $\Gamma(A, I)$,  for some nontrivial ideal $I$ of $A$. Note that $I\subseteq \ppp$. In 
fact, if $a \in I$, then ${\mtxx 1 a 0 1} \in \Gamma(A,I) \subseteq \Gamma(A, \ppp)$. It 
follows from this that $a \in \ppp$. Since $\Gamma(A, 2^23^2I)\se \Gamma(A, I)$, we may 
assume that $I=\lan 2^23^2y\ran$, for some $y\in A$.

Let $I = \lan p_1^{r_1} \cdots p_m^{r_m}\ran$, where $p_1=p$ and $p_2 < p_3 <\cdots < p_m$ are primes, and 
so irreducible, elements of $A$. Let $k:= r_1$ and $J := \lan p_2^{r_2} \cdots p_m^{r_m}\ran$. Then $I=\ppp^k J$, 
where $k\geq 1$. Note that $\ppp^k$ and $J$ are coprime. Since $\Gamma(A, \ppp^{k+1}) \subseteq \Gamma(A, \ppp^{k})$, 
we may assume without loss that $k\geq 2$. 
%\textcolor{red}{(Remark 9) In fact, the ideals $\ppp$ and $J$ are coprime and, applying directly Lemma \ref{coprime}, we would come to the same result. We can write in this way, but the thing is that there are interesting arguments in the original way, which also appear in Hutchinson's paper.} 
Now, by Lemma \ref{coprime},
%\textcolor{blue}{(Objection 9: He might be have a good point since $\ppp$ is coprime with $J$. Indeed, I guess that we can rewrite this fragment as:
%Note that $\ppp^k$ and $J$ are coprimes, hence $\ppp$ and $J$ are also. Now, by Lemma \ref{coprime} we have that the natural map: $\Gamma(A,\ppp) \rightarrow \SL_2(A/J)$ is surjective.)}, 
the map
\[
\Gamma(A, \ppp) \arr \SL_2(A/J)
\]
is surjective. 
%Since $\Gamma(A, \ppp^{k})\subseteq \Gamma(A, \ppp)$, it follows that the natural map 
%$\Gamma(A, \ppp) \arr \SL_2(A/J)$ is surjective. 
From this, we obtain the surjective map 
%\textcolor{blue}{((Objection 10): He comments about the choice for J. Now it is clear by our last paper the relation between $H_1(\SL_2(A))$ with $H_1(\SL_2(A/J)$. Maybe it is good idea to refer the result here as can be useful to clarify what are happen to the reader.)}
\[
[\Gamma(A, \ppp), \Gamma(A, \ppp)] \arr [\SL_2(A/J), \SL_2(A/J)].
\]
Moreover, the inclusion $I\subseteq J$ gives us the inclusion $\Gamma(A, I) \subseteq \Gamma(A,J)$. From this, we obtain the 
surjective map
\[
[\Gamma(A, \ppp), \Gamma(A, \ppp)]/\Gamma(A,I) \arr [\SL_2(A/J), \SL_2(A/J)]
\]
and thus
\begin{equation}\label{eq1}
|[\SL_2(A/J), \SL_2(A/J)]|\ \Big|\ [[\Gamma(A, \ppp), \Gamma(A, \ppp)]:\Gamma(A,I)].
\end{equation}
%Since 
%\[
%A/J\simeq A/\lan p_2\ran^{r_2}\times \cdots \times A/\lan p_m\ran^{r_m},
%\]
%we have
%\[
%H_1(\SL_2(A/J),\z) \simeq H_1(\SL_2(A/\lan p_2\ran^{r_2}),\z) \times \cdots \times 
%H_1(\SL_2(A/\lan p_m\ran^{r_m}),\z).
%\]
%If $2\mid n$ and $3\mid n$, then $2,3\in \aa$. Thus, for any $2 \leq i \leq m$, $p_i>3$ and so by 
%the above isomorphism and Proposition~\ref{exa-2}, we have $H_1(\SL_2(A/J),\z)=0$. By a similar 
%argument, we can show that 
%\[
%\begin{array}{c}
%H_1(\SL_2(A/J),\z) \simeq 
%\end{array}
%\begin{cases}
%0   & \text{if $2\mid n$, $3\mid n$}  \\
%\z/3   & \text{if $2\mid n$, $3 \nmid n$}\\
%\z/4   & \text{if $2\nmid n$, $3\mid n$}\\
%\z/12 &  \text{if $2\nmid n$, $3\nmid n$}
%\end{cases}.
%\]
%Let $l:=|H_1(\SL_2(A/J),\z)|$. By Theorem \ref{H1}, observe that $l = |H_1(\SL_2(A),\z)|$. 
From the isomorphism 
\[
H_1(\SL_2(A/J),\z)\simeq \displaystyle\frac{\SL_2(A/J)}{[\SL_2(A/J), \SL_2(A/J)]},
\]
we have
\[
|[\SL_2(A/J), \SL_2(A/J)]|=|\SL_2(A/J)|/|H_1(\SL_2(A/J),\z)|.
\] 
It follows from this and (\ref{eq1}) that
\[
|\SL_2(A/J)|/|H_1(\SL_2(A/J),\z)|\ \Big|\ [[\Gamma(A, \ppp), \Gamma(A, \ppp)]:\Gamma(A,I)].
\]
Let $m \in \N$ such that
\[
[[\Gamma(A, \ppp), \Gamma(A, \ppp)]:\Gamma(A,I)]=\frac{m |\SL_2(A/J)|}{|H_1(\SL_2(A/J),\z)|}.
\]
Since $\ppp^k$ and $J$ are coprime ideals of $A$ and $I=\ppp^kJ$, we have 
$A/I \simeq A/\ppp^k \times A/J$ and, thus, 
\begin{equation}\label{eq2}
\SL_2(A/I)\simeq \SL_2(A/\ppp^k) \times \SL_2(A/J).
\end{equation}
By Theorem \ref{cong}, the index 
$[\SL_2(A): [\Gamma(A, \ppp),\Gamma(A, \ppp)]]$ is finite and so we have
\begin{align*}
[\SL_2(A): [\Gamma(A, \ppp),\Gamma(A, \ppp)]]& =\frac{[\SL_2(A): \Gamma(A,I)]}{[[\Gamma(A, \ppp),\Gamma(A, \ppp)]: \Gamma(A,I)]}\\
&=\frac{|\SL_2(A/I)|}{[[\Gamma(A, \ppp),\Gamma(A, \ppp)]: \Gamma(A,I)]}\\
&=\frac{|\SL_2(A/\ppp^k)||\SL_2(A/J)|}{[[\Gamma(A, \ppp),\Gamma(A, \ppp)]: \Gamma(A,I)]}\\
&= \frac{|H_1(\SL_2(A/J),\z)||\SL_2(A/\ppp^k)|}{m}.
\end{align*}
By (\ref{eq2}), $H_1(\SL_2(A/I), \z)\simeq H_1(\SL_2(A/J), \z) \oplus H_1(\SL_2(A/\ppp^k), \z)$. Hence,
\begin{align*}
\frac{|H_1(\SL_2(A/I),\z)||\SL_2(A/\ppp^k)|}{|H_1(\SL_2(A/\ppp^k),\z)|}&=m[\SL_2(A): 
[\Gamma(A, \ppp),\Gamma(A, \ppp)]]\\
&=m[\SL_2(A): \Gamma(A, \ppp)][\Gamma(A, \ppp): [\Gamma(A, \ppp),\Gamma(A, \ppp)]]\\
&\overset{(\ref{eq0})}{=}mp(p^2-1)|H_1(\Gamma(A, \ppp),\z)|\\
&=mp(p^2-1)|\KK||\GG|.
\end{align*}
Now we prove that 
\begin{equation}\label{eq3}
H_1(\SL_2(A/I),\z)\simeq H_1(\SL_2(A),\z).
\end{equation}
Since 
\[
\Gamma(A,I)\se [\Gamma(A, \ppp),\Gamma(A, \ppp)] \se [\SL_2(A), \SL_2(A)],
\]
we have the surjective map
\[
\phi:\SL_2(A/I)\simeq \SL_2(A)/\Gamma(A,I)  \two    \SL_2(A)/[\SL_2(A), \SL_2(A)]=H_1(\SL_2(A),\z).
\]
This gives us the surjective map
\[
\phi_\ast: H_1(\SL_2(A/I),\z) \arr H_1(\SL_2(A),\z).
\]
On the other hand, we have the natural surjective map
\[
\pi_\ast: H_1(\SL_2(A),\z) \two H_1(\SL_2(A/I),\z).
\]
It is straightforward to check that the maps $\phi_\ast$ and $\pi_\ast$ are inverse of each other.
This proves (\ref{eq3}). 

Let $l=|H_1(\SL_2(A),\z)|$. By Theorem \ref{H1} we have
\[
H_1(\SL_2(A/\ppp^k),\z)\simeq \begin{cases}
0 & \text{if $p>3$}\\
\z/3 & \text{if $p=3$.}\\
\z/4 & \text{if $p=2$}
\end{cases}
\]
Hence, $|H_1(\SL_2(A/\ppp^k),\z)|=p^\alpha$, for some $\alpha \in \{0, 1, 2\}$. Now from the above we have 
\[
l|\SL_2(A/\ppp^k)|=mp(p^2-1)|\KK||\GG||H_1(\SL_2(A/\ppp^k),\z)|=mp(p^2-1)|\KK||\GG|p^\alpha.
\]
Using diagrams (\ref{diag1}), (\ref{diag--1}),  (\ref{diag-1}),  (\ref{diag--11}),  (\ref{diag-11}) and 
Theorem~\ref{H1} and Proposition~\ref{exa-2}, we easily get
\[
|\GG|=\begin{cases}
l & \text{if $p>3$}\\
l/3 & \text{if $p=3$.}\\
l/2 & \text{if $p=2$}
\end{cases}
\]
Thus, $l=p^r|\GG|$, where $r\in \{0,1\}$.
Now, since $|\SL_2(A/\ppp^k)|=(p^2-1)p^{3k-2}$ (see \cite[p. 248]{h2016}), we have 
\[
l(p^2-1)p^{3k-2}=mp(p^2-1)|\KK||\GG|p^\alpha
\]
and hence
\[
p^{3k-3+r}=m|\KK|p^\alpha.
\]
It follows from this that $|\KK|$ is a power of $p$.
This completes the proof of the theorem.
\end{proof}

The next theorem is Theorem C of the introduction.

\begin{thm}\label{G0}
Let $n>1$ be a square-free natural number, $p$ a prime such that $p\nmid n$ and $d:= \gcd\{m^2-1: m\mid n\}$. 
\par {\rm (i)} If $p>3$ and $p\nmid d$, then
\[
\begin{array}{c}
H_1(\Gamma_0(n,p),\z)\simeq  H_1(\SL_2(\z[1/n]),\z) \oplus \F_p^\times.
\end{array}
\]
\par {\rm (ii)} If $p=3$ and $d=3t$, where $3\nmid t$ (e.g. when $2\mid n$ or $5\mid n$), then 
\[
\begin{array}{c}
H_1(\Gamma_0(n, 3),\z)\simeq  H_1(\SL_2(\z[1/n]),\z) \oplus \F_3^\times \oplus \z/3.
\end{array}
\]
\par {\rm (iii)} If $p=2$ and $d=8t$, where $2\nmid t$ (e.g. when $3\mid n$ or $5\mid n$), then 
\[
\begin{array}{c}
H_1(\Gamma_0(n,2),\z)\simeq  H_1(\SL_2(\z[1/n]),\z) \oplus
\F_2 \oplus \z/4.
\end{array}
\]
\end{thm}
\begin{proof}
Let $A:=\z[1/n]$. The map $\Gamma_0(n,p) \arr \F_p^\times$, given by ${\mtxx a b c d} \mapsto \overline{a}$,
is a surjective homomorphism of groups. We denote its kernel by $\Gamma_1(n,p)$. Thus, we have the group extension
\[
1 \arr \Gamma_1(n,p)\arr \Gamma_0(n,p) \arr \F_p^\times \arr 1,
\]
where
\[
\begin{array}{c}
\Gamma_1(n,p):=\bigg\{{\mtxx a b c d} \in\SL_2(A): p\mid a-1, d-1, c  \bigg\}.
\end{array}
\]
Note that $\Gamma_1(n,p)=\widetilde{\Gamma}(A, \ppp)$, where $\ppp=\lan p\ran$, and
\[
\begin{array}{c}
\Gamma(n,p) \subseteq \Gamma_1(n,p) \subseteq\Gamma_0(n,p) \subseteq \SL_2(A).
\end{array}
\]
By Theorem \ref{VL} (for $I_1=A$ and $I_2=\ppp)$, $\Gamma_1(n,p)$ is generated by 
the matrices $E_{12}(x)$ and $E_{21}(py)$, with $x, y\in A$. From the morphism of extensions 
\[
\begin{tikzcd}
1 \ar[r] & \Gamma(n, p) \ar[r]\ar[d, hookrightarrow]& \Gamma_0(n,p) \ar[r]\ar[d, equal]& \B(\F_p) \ar[r] \ar[d, "p"]& 1\\
1 \ar[r] & \Gamma_1(n, p) \ar[r]& \Gamma_0(n,p) \ar[r]& \F_p^\times \ar[r]& 1,
\end{tikzcd}
\]
we obtain the commutative diagram with exact rows 
\[
\begin{tikzcd} 
0 \ar[r]& H_1(\Gamma(n,p),\z)_{\B(\F_p)} \ar[r]\ar[d]& H_1(\Gamma_0(n,p),\z) \ar[r] \ar[d, equal]& 
\F_p^\times\oplus (\F_p)_{\F_p^\times}\ar[r] \ar[d, "{p}_\ast"]& 1\\
0 \ar[r]& H_1(\Gamma_1(n,p),\z)_{\F_p^\times}\ar[r]& H_1(\Gamma_0(n,p),\z) \ar[r]& \F_p^\times\ar[r]& 1.
\end{tikzcd}
\]
Since 
\[
H_1(\Gamma(n,p),\z)_{\B(\F_p)}\simeq\Gamma(n,p)/[\Gamma(n,p), \Gamma_0(n,p)]
\]
and
\[
H_1(\Gamma_1(n,p),\z)_{\F_p^\times}\simeq\Gamma_1(n,p)/[\Gamma_1(n,p), \Gamma_0(n,p)],
\]
from the Snake lemma applied to the above diagram,  we obtain the exact sequence
\[
1 \arr \frac{\Gamma(n,p)}{[\Gamma(n,p),\Gamma_0(n,p)]} \arr \frac{\Gamma_1(n,p)}{[\Gamma_1(n,p), \Gamma_0(n,p)]} 
\arr (\F_p)_{\F_p^\times} \arr 0.
\]
Again, applying the Snake lemma to the following commutative diagram 
\[
\begin{tikzcd}
1  \ar[r] & \displaystyle\frac{\Gamma(n,p)}{[\Gamma(n,p),\Gamma_0(n,p)]} \ar[r]\ar[d] & 
\displaystyle\frac{\Gamma_1(n,p)}{[\Gamma_1(n,p), \Gamma_0(n,p)]} \ar[r]\ar[d, two heads] & 
(\F_p)_{\F_p^\times} \ar[r]& 0\\
& H_1(\SL_2(A),\z) \ar[r, equal] & H_1(\SL_2(A),\z), & & 
\end{tikzcd}
\]
we obtain the exact sequence
\[
0 \arr \KK \arr \LL \arr (\F_p)_{\F_p^\times},
\]
where $\KK$ and $\LL$ are the kernels of the left and right vertical maps, respectively. 
By Theorem~\ref{p-group}, $\KK$ is a $p$-group. Thus, by the above exact sequence, 
$\LL$ is also a $p$-group. The natural map
\[
\begin{array}{c}
\tau: A \times A 
\end{array}
\arr \frac{\Gamma_1(n,p)}{[\Gamma_1(n,p), \Gamma_0(n,p)]},
\]
defined by $(a,0) \mapsto \overline{E_{12}(a)}$ and  $(0, a) \mapsto \overline{E_{21}(pa)}$,
is a surjective map. Since $d=\gcd\{m^2-1: m\mid n\}$, there are $r_m \in A$ such that 
$d=\sum_{m\mid n}r_m(m^2-1)$. Thus
\begin{align*}
\tau(-d,0) &= \overline{E_{12}(-d)}=\overline{E_{12}\Bigg(\sum_{m\mid n}r_m(1-m^2)\Bigg)}
=\prod_{m\mid n}\overline{E_{12}(r_m(1-m^2))}\\
&=\prod_{m\mid n} \overline{[D(m^{-1}), E_{12}(r_m)]}=1,
\end{align*}
where $D(m):={\mtxx m 0 0 {1/m}} \in \Gamma_0(n,p)$. Similarly,
$\tau(0,-d) =\prod_{m\mid n} \overline{[D(m), E_{21}(pr_m)]}=1$.
Thus, we have the surjective map
\[
\begin{array}{c}
\overline{\tau}:A/\lan d \ran \times A/\lan d \ran \arr 
\displaystyle\frac{\Gamma_1(n,p)}{[\Gamma_1(n,p),\Gamma_0(n,p)]}.
\end{array}
\]
Since $A/\lan d \ran \simeq \z/d$, we have
$\displaystyle\Bigg|\frac{\Gamma_1(n,p)}{[\Gamma_1(n,p), \Gamma_0(n,p)]}\Bigg| \mid d^2$. Hence,
\begin{align}\label{divides}
l \cdot |\LL| \ \Big|\ d^2,
\end{align}
where $l=|H_1(\SL_2(A), \z)|$.

\par (i) Let $p>3$. Then, on the one hand, $\LL$ is a $p$-group and, on the other hand, $|\LL|$ 
divides $d^2$. Since $p\nmid d$, $\LL$ must be trivial. It follows from this that 
\[
\frac{\Gamma_1(n,p)}{[\Gamma_1(n,p),\Gamma_0(n,p)]}
\begin{array}{c}
\simeq H_1(\SL_2(A), \z).
\end{array}
\]
Now, from the commutative diagram with exact rows
\[
\begin{tikzcd}
1  \ar[r] & \displaystyle\frac{\Gamma_1(n,p)}{[\Gamma_1(n,p),\Gamma_0(n,p)]} \ar[r]\ar[d, "\simeq"] & 
H_1(\Gamma_0(n,p),\z) \ar[r]\ar[d, two heads] & \F_p^\times \ar[r]& 1\\
& H_1(\SL_2(A),\z) \ar[r, equal] & H_1(\SL_2(A),\z), & & 
\end{tikzcd}
\]
we see that the first row splits. This completes the proof of the first item. 

(ii) Let $p=3$. Since $3\nmid n$, 
\[
l=|H_1(\SL_2(A),\z)|=\begin{cases}
3 & \text{if $2 \mid n$}\\
12 & \text{if $2 \nmid n$}
\end{cases}=2^r3, \ \ \ \ \ r=0,2,
\]
(see Theorem \ref{H1}). Thus, by (\ref{divides}), we get
\[
2^r3|\LL|\ \Big|\ d^2=3^2t^2.
\]
Since $\LL$ is a $3$-group (Theorem~\ref{p-group}) and $3\nmid t$, we have $|\LL|\ \Big|\ 3$. 
Under the natural map $\displaystyle\frac{\Gamma_1(n,3)}{[\Gamma_1(n,3),\Gamma_0(n,3)]}
\overset{{i_1}_\ast}{-\!\!\!-\!\!\!\larr} H_1(\SL_2(A),\z)$, 
$\overline{E_{21}(l)}$ maps to zero. In fact, in $ H_1(\SL_2(A),\z)$, we have 
\[
\overline{E_{21}(l)}=\overline{wE_{21}(l)w^{-1}}=\overline{E_{12}(-l)}=1,
\]
where $w={\mtxx 0 1 {-1} 0}$ (see Theorem \ref{H1}). But, under the map 
\[
\frac{\Gamma_1(n,3)}{[\Gamma_1(n,3),\Gamma_0(n,3)]} 
\overset{({i_1}_\ast,{i_2'}_\ast)}{-\!\!\!-\!\!\!-\!\!\!-\!\!\!\larr}
H_1(\SL_2(A), \z) \oplus H_1(\SL_2(A), \z),
\]
we have
\[
({i_1}_\ast,{i_2'}_\ast)(\overline{E_{21}(l)}))=(\overline{E_{21}(l)},\overline{E_{21}(l/3)})
=(1,\overline{E_{21}(l/3)})=(1,\overline{E_{12}(l/3)})\neq 1.
\]
Thus $\overline{E_{21}(l)})$ is a nonzero element of $\displaystyle\frac{\Gamma_1(n,3)}{[\Gamma_1(n,3),\Gamma_0(n,3)]}$
that belongs to $\LL$. Thus $\LL$ is non-trivial and hence, $\LL\simeq \z/3$. Therefore, we have the exact sequence
\[
0 \arr \z/3 \arr \frac{\Gamma_1(n,3)}{[\Gamma_1(n,3),\Gamma_0(n,3)]} \arr H_1(\SL_2(A), \z)\arr 0.
\]
Using Theorem \ref{H1} it is straightforward to check that the sequence
\[
\displaystyle\frac{\Gamma_1(n,3)}{[\Gamma_1(n,3),\Gamma_0(n,3)]} 
\overset{({i_1}_\ast,{i_2'}_\ast)}{-\!\!\!-\!\!\!-\!\!\!-\!\!\!\larr}
\begin{array}{c}
H_1(\SL_2(A), \z) \oplus H_1(\SL_2(A), \z) \arr H_1(\SL_2(\z[1/3n]), \z) \arr 0
\end{array}
\]
is exact. Thus, the above exact sequence splits and hence,
\[
\frac{\Gamma_1(n,3)}{[\Gamma_1(n,3),\Gamma_0(n,3)]}\simeq  H_1(\SL_2(A), \z) \oplus \z/3.
\]
Therefore,
\[
\begin{array}{c}
H_1(\Gamma_0(n,3),\z)\simeq \F_3^\times \oplus H_1(\SL_2(A), \z)\oplus \z/3.
\end{array}
\]
This completes the proof of (ii).

(iii) Let $p=2$. Note that, in this case, $\Gamma_1(n,2)=\Gamma_0(n,2)$. Then
\[
\begin{array}{c}
H_1(\Gamma_0(n,2),\z)\simeq 
\end{array}
\frac{\Gamma_1(n,2)}{[\Gamma_1(n,2),\Gamma_0(n,2)]}.
\]
Since $2\nmid n$,
\[
l=|H_1(\SL_2(A),\z)|=\begin{cases}
4 & \text{if $3 \mid n$}\\
12 & \text{if $3 \nmid n$}
\end{cases}=2^23^r, \ \ \ \ \ r=0,1,
\]
(see Theorem \ref{H1}). Thus, by (\ref{divides}), we get
\[
2^23^r|\LL|\ \Big|\ d^2=(8t)^2.
\]
Since $\LL$ is a $2$-group and $2\nmid t$, $|\LL|\ \Big |\ 16$. In $\Gamma_0(n,2)$ we have
\[
-I_2=E_{21}(-2)E_{12}(1)E_{21}(-2) E_{12}(1).
\]
Hence, in $H_1(\Gamma_0(n,2),\z)=\displaystyle\frac{\Gamma_0(n,2)}{[\Gamma_0(n,2),\Gamma_0(n,2)]}$, 
we have
\[
1=\overline{I_2}=(\overline{-I_2})^2=\overline{E_{21}(-8)}\ \overline{E_{12}(4)}.
\]
This implies that under the map
\[
\z/8 \times \z/8 \arr \frac{\Gamma_0(n,2)}{[\Gamma_0(n,2),\Gamma_0(n,2)]}, 
\]
the element $(4,-4)$ maps to zero. It follows from this that $|\LL|\ \Big|\ 8$. Since we have the 
exact sequence
\[
\frac{\Gamma_0(n,2)}{[\Gamma_0(n,2),\Gamma_0(n,2)]}\arr H_1(\SL_2(A),\z)\oplus H_1(\SL_2(A),\z)\arr 
\begin{array}{c}
H_1(\SL_2(\z[1/2n]), \z) \arr 0,
\end{array}
\]
using Theorem \ref{H1}, we get the following estimate on the order of $H_1(\Gamma_0(n,2),\z)$:
\[
16 \leq | H_1(\Gamma_0(n,2),\z)| \leq 32.
\]
But the element $\overline{E_{21}(8)}=\overline{E_{12}(4)}$ of 
$\displaystyle\frac{\Gamma_0(n,2)}{[\Gamma_0(n,2),\Gamma_0(n,2)]}$ 
is of order $2$ and, under the above map, goes to zero. This is a
nontrivial element of $H_1(\Gamma_0(n,2),\z)$ and thus 
\[
|H_1(\Gamma_0(n,2),\z)|=32.
\]
Now, as in case of (ii), we can show that
\[
\begin{array}{c}
H_1(\Gamma_0(n,2),\z)\simeq H_1(\SL_2(A), \z)\oplus \F_2 \oplus \z/4.
\end{array}
\]
This completes the proof of (iii) and the proof of the theorem.
\end{proof}

\begin{rem}
We believe that the theorem holds even after removing the conditions placed on $d$ 
(the divisibility restrictions in each case). However, the current theorem, with those 
conditions, is sufficient to prove Theorem \ref{exact2}, which relates to the group 
structure of $H_2(\SL_2(\z[1/n]),\z)$.
\end{rem}

%%%%%%%%%%%%%%%%%%%%%%%%%%%%%%%%%%%%%%%%%%%%%%%%%%%%%%%%%%%%%%%%%%%%%%%%%%%%%%%%%%
\section{The second homology of \texorpdfstring{$\Gamma_0(n,p)$}{Lg}}\label{sec5}
%%%%%%%%%%%%%%%%%%%%%%%%%%%%%%%%%%%%%%%%%%%%%%%%%%%%%%%%%%%%%%%%%%%%%%%%%%%%%%%%%%

We now state the following lemma. Here, for a finite abelian group $N$, $N_{(p)}$ denotes the $p$-Sylow 
subgroup of $N$.

\begin{lem}\label{p-torsion}
For any $\SL_2(\F_p)$-module $M$ and any integer $m\geq 1$, we have the isomorphism
\[
H_m(\B(\F_p), M)_{(p)} \simeq H_m(\SL_2(\F_p), M)_{(p)}.
\]
\end{lem}
\begin{proof}
See \cite[Lemma 5.15]{h2016}.
\end{proof}

\begin{thm}\label{H2-surj}
Let $n>1$ be a square-free integer and $p$ a prime such that $p\nmid n$. Then the natural maps 
\[
\begin{array}{c}
{i_1}_\ast: H_2(\Gamma_0(n,p),\z) \arr H_2(\SL_2(\z[1/n]),\z),
\end{array}
\]
\[
\begin{array}{c}
{i_2'}_\ast: H_2(\Gamma_0(n,p),\z) \arr H_2(\SL_2(\z[1/n]),\z)
\end{array}
\]
are surjective. 
\end{thm}
\begin{proof}
We first prove the claim for ${i_1}_\ast$. 
The morphism of extensions (\ref{ext}) gives us the morphism
of spectral sequences
\[
\begin{tikzcd}
{E'}_{r,s}^2=H_r(\B(\F_p), H_s(\Gamma(n,p),\z)) \ar[r,Rightarrow] \ar[d]& H_{r+s}(\Gamma_0(n,p),\z)\ar[d]\\
E_{r,s}^2=H_r(\SL_2(\F_p), H_s(\Gamma(n,p),\z))\ar[r,Rightarrow]& H_{r+s}(\SL_2(\z[1/n]),\z).
\end{tikzcd}
\]
Since $H_2(\SL_2(\F_p),\z)=0$ and $H_2(\B(\F_p),\z)=0$ (see Section \ref{sec3}), we have $E_{2,0}^2=0$ and
${E'}_{2,0}^2=0$. By an easy analysis of these spectral sequences, we obtain the commutative diagram with exact rows
\[
\begin{tikzcd}
H_2(\Gamma(n,p),\z) \ar[r] \ar[d, equal] & H_2(\Gamma_0(n,p),\z) \ar[r] \ar[d] &{E'}_{1,1}^\infty \ar[r]\ar[d] & 0\\
H_2(\Gamma(n,p),\z) \ar[r] & H_2(\SL_2(\z[1/n]),\z) \ar[r] & E_{1,1}^\infty \ar[r] & 0.
\end{tikzcd}
\]
So, to prove the claim, we may show that the map ${E'}_{1,1}^\infty \arr E_{1,1}^\infty$ is 
surjective. Again, from the above morphism of spectral sequences, we obtain the commutative 
diagram with exact rows
\[
\begin{tikzcd}
H_1(\B(\F_p),H_1(\Gamma(n,p),\z))={E'}_{1,1}^2 \ar[r, two heads] \ar[d] & {E'}_{1,1}^\infty\ar[d]\\
H_1(\SL_2(\F_p),H_1(\Gamma(n,p),\z))=E_{1,1}^2 \ar[r, two heads] & E_{1,1}^\infty .
\end{tikzcd}
\]
So, to prove the surjectivity of the right vertical map, it is sufficient to prove the surjectivity 
of the left vertical map. Let $\KK_p$ and $\GG_p$ be the kernel and the image of the natural map
\[
\begin{array}{c}
i_\ast: H_1(\Gamma(n,p),\z) \arr H_1(\SL_2(\z[1/n]),\z),
\end{array}
\]
respectively. So, we have the exact sequence
\[
0 \arr \KK_p \arr H_1(\Gamma(n,p),\z) \arr \GG_p \arr 0.
\]
By Theorem~\ref{p-group}, $\KK_p$ is a $p$-group. By studying the diagrams (\ref{diag1}), (\ref{diag--1}),  
(\ref{diag-1}),  (\ref{diag--11}) and  (\ref{diag-11}), we see that 
\[
\begin{array}{c}
\GG_p\simeq H_1(\SL_2(\z[1/n]),\z), \ \ \ \ \ \ \ \text{for $p>3$},
\end{array}
\]
and 
\[
\GG_3\simeq 
\begin{cases}
0 & \text{if $2\mid n$}  \\ 
\z/4 & \text{if $2\nmid n$}
\end{cases}, \ \ \ \ \ \ \ 
\GG_2\simeq 
\begin{cases}
\z/2 & \text{if $3\mid n$}  \\ 
\z/6 & \text{if $3\nmid n$}
\end{cases}.
\]
It follows from these that 
\[
H_1(\Gamma(n,p),\z)\simeq 
\begin{cases}
\KK_p \oplus \GG_p & \text{if $p>3$}  \\ 
\KK_3  & \text{if $p=3$, $2\mid n$}  \\
\KK_3 \oplus \z/4 & \text{if $p=3$, $2\nmid n$}  \\
\KK_2' \oplus \z/3 & \text{if $p=2, 3\nmid n$}\\
\KK_2' & \text{if $p=2, 3\mid n$}
\end{cases},
\]
where $\KK_2'$ is a $2$-group.
%and $H_1(\Gamma(n,2),\z)$ is a $2$-group if $3\mid n$. 
All these show that, for any prime $p$,
\[
H_1(\Gamma(n,p),\z)\simeq \KK_p'' \oplus \GG_p'',
\]
where $\KK_p''$ is a $p$-group and $\GG_p''$ is a subgroup of $H_1(\SL_2(\z[1/n]),\z)$ 
such that $p\nmid |\GG_p''|$. Now we are ready to study the map
\begin{equation}\label{ttt}
H_1(\B(\F_p), H_1(\Gamma(n,p),\z)) \arr H_1(\SL_2(\F_p), H_1(\Gamma(n,p),\z)).
\end{equation}
From this and the above isomorphism, we obtain the map
\[
H_1(\B(\F_p),\KK_p'')\oplus H_1(\B(\F_p),\GG_p'')\arr  
H_1(\SL_2(\F_p),\KK_p'')\oplus H_1(\SL_2(\F_p),\GG_p'').
\]
The induced map 
\[
H_1(\B(\F_p),\KK_p'')\arr  H_1(\SL_2(\F_p),\KK_p'')
\]
is a map of $p$-groups (see \cite[Corollary 11.8.12]{v2003}), so by Lemma \ref{p-torsion} it is an isomorphism.
Since $\SL_2(\F_p)$ acts trivially on the group $H_1(\SL_2(\z[1/n]),\z)$, it acts trivially 
on $\GG_p''$. Now, by the Universal Coefficient Theorem, we have
\begin{align*}
H_1(\SL_2(\F_p), \GG_p'') \simeq H_1(\SL_2(\F_p), \z)\otimes_\z \GG_p''.
\end{align*}
By Proposition \ref{exa-2},
$
H_1(\SL_2(\F_p), \z)\simeq \begin{cases}
0 & \text{if $p>3$}\\
\z/3 & \text{if $p=3$}\\
\z/2 & \text{if $p=2$}
 \end{cases}$.
This shows that the order of $H_1(\SL_2(\F_p), \z)$ and the order of $\GG_p''$ are coprime. Hence,
\[
H_1(\SL_2(\F_p), \GG_p'')=0.
\]
All these imply that the map (\ref{ttt}) is surjective. This completes the proof of the surjectivity of 
${i_1}_\ast$. 

Now consider the map ${i_2'}_\ast$. We remind that $i_2': \Gamma_0(n,p) \harr \SL_2(\z[1/n])$ is given 
by ${\mtxx a b c d}\mapsto {\mtxx a {pb} {p^{-1}c} d}$. Let $\Gamma_0'(n,p)$ be the image of $i_2'$. 
Thus
\[
\begin{array}{c}
\Gamma_0'(n, p):=\im(i_2')=\bigg\{{\mtxx a b c d}\in \SL_2(\z[1/n]) : p\mid b\bigg\}.
\end{array}
\]
Let $i_2''$ denotes the inclusion $\Gamma_0'(n, p) \harr \SL_2(\z[1/n])$. So, to prove the surjectivity of 
${i_2'}_\ast$, it is sufficient to prove the surjectivity of ${i_2''}_\ast$.

Now the proof of the surjectivity of ${i_2''}_\ast$ follows a similar pass as the proof of the surjectivity of 
${i_1}_\ast$. Here, we have to study the Lyndon/Hochschild-Serre spectral sequence associated to the morphism of
extensions
\begin{equation*}
\begin{tikzcd}
1 \ar[r] & \Gamma(n,p) \ar[r]\ar[d, equal]& \Gamma_0'(n,p) \ar[r, "{\pi'}"]\ar[d, "{i_2''}"]& \B'(\F_p) \ar[r] \ar[d, hook]& 1\\
1 \ar[r] & \Gamma(n,p) \ar[r]& \SL_2(\z[1/n]) \ar[r]& \SL_2(\F_p) \ar[r]& 1,
\end{tikzcd}
\end{equation*}
where
\[
\B'(\F_p):=\bigg\{ {\mtxx x 0 y {x^{-1}}}\in \SL_2(\F_p): x\in \F_p^\times, y\in \F_p\bigg\} \ \ \text{and} \ \ \
\pi'({\mtxx a b c d})={\mtxx {\overline{a}} 0 {\overline{c}} {\overline{a}^{-1}}}.
\]
\end{proof}

%%%%%%%%%%%%%%%%%%%%%%%%%%%%%%%%%%%%%%%%%%%%%%%%%%%%%%%%%%%%%%%%%%%%%%%%%%%%%%%%%%
\section{The second homology of \texorpdfstring{$\SL_2(\z[1/n])$}{Lg}}\label{sec6}
%%%%%%%%%%%%%%%%%%%%%%%%%%%%%%%%%%%%%%%%%%%%%%%%%%%%%%%%%%%%%%%%%%%%%%%%%%%%%%%%%%

\begin{prp}\label{exact1}
Let $n>1$ be a square-free integer, $p$ a prime such that $p\nmid n$ and $d:=\gcd\{m^2-1:m\mid n\}$. 
\par {\rm (i)} If $p>3$ and $p\nmid d$, then we have the exact sequence
\[
\begin{array}{c}
H_2(\SL_2(\z[1/n]),\z) \arr H_2(\SL_2(\z[1/pn]),\z) \arr \F_p^\times \arr 1.
\end{array}
\]
\par {\rm (ii)} If $p=3$ and $d=3t$, where $3\nmid t$ (e.g. when $2\mid n$ or $5\mid n$), then
we have the exact sequence
\[
\begin{array}{c}
H_2(\SL_2(\z[1/n]),\z) \arr H_2(\SL_2(\z[1/3n]),\z) \arr \F_3^\times \arr 1.
\end{array}
\]
\par {\rm (iii)} If $p=2$ and $d=8t$, where $2\nmid t$ (e.g. when $3\mid n$ or $5\mid n$), 
then we have the exact sequence
\[
\begin{array}{c}
H_2(\SL_2(\z[1/n]),\z) \arr H_2(\SL_2(\z[1/2n]),\z) \arr \F_2 \arr 0.
\end{array}
\]
\end{prp}
\begin{proof}
These follow from the Mayer-Vietoris exact sequence (see Section \ref{sec3}), Theorem~\ref{G0} and Theorem~\ref{H2-surj}.
\end{proof}

Let $n>1$ be an integer and let $p$ be a prime such that $p\nmid n$. Let $\delta_p$ be the composition
\[
H_2(\SL_2(\z[1/pn]),\z) \arr H_1(\Gamma_0(n,p),\z) \two H_1(\B(\F_p),\z)\two \F_p^\times.
\]

\begin{lem}\label{H2-to-Fp}
If $n>1$ is a square-free integer and $p>3$ is a prime such that $p\nmid n$, then the map
\[
\begin{array}{c}
\delta_p: H_2(\SL_2(\z[1/pn]),\z) \arr \F_p^\times
\end{array}
\]
is surjective. 
\end{lem}
\begin{proof}
By Theorem \ref{H1}, $H_1(\SL_2(\z[1/n]),\z)\simeq H_1(\SL_2(\z[1/pn]),\z)$.
Thus, from the Mayer-Vietoris exact sequence, we get the exact sequence
\[
\begin{array}{c}
H_2(\SL_2(\z[1/pn]),\z) \arr H_1(\Gamma_0(n,p),\z) \overset{{i_1}_\ast}{\larr} 
H_1(\SL_2(\z[1/n]),\z)\arr 0.
\end{array}
\]
By applying the Snake lemma to the diagram (\ref{diag1}), we obtain the exact sequence
\[
1 \arr \frac{[\Gamma(n,p), \SL_2(\z[1/n])]}{[\Gamma(n,p), \Gamma_0(n,p)]} \arr 
\ker({i_1}_\ast) \arr \F_p^\times \arr 1.
\]
Now, the composition
\[
\begin{array}{c}
H_2(\SL_2(\z[1/pn]),\z) \two \ker({i_1}_\ast) \two \F_p^\times 
\end{array}
\]
is surjective and coincides with the above composition. 
\end{proof}

\begin{rem}
We believe that the above lemma is correct for $p = 3$ and $p = 2$  
(with $\F_2^\times$ replaced by $\F_2$ in the case $p = 2$).  
Since this result was not required, we did not attempt a proof.
\end{rem}

%Let $n>1$ be an integer and $p$ a prime such that $p\nmid n$. We denote the natural map 
%$H_2(\SL_2(\z[1/pn]),\z)\arr \F_p^\times$ by $\delta_p$. \textcolor{red}{}

If $p_1,\dots, p_k$, $p_i>2$, are distinct primes such that $p_i\nmid n$, then the maps
$\delta_{p_i}$ induce the natural map
\[
\begin{array}{c}
\delta_{p_1, \dots,p_k}:=(\delta_{p_i})_{i=1}^k:H_2(\SL_2(\z[{1}/{(p_1\cdots p_kn)}]),\z) \arr 
\bigoplus_{i=1}^k \F_{p_i}^\times.
\end{array}
\]

Let $F$ be a field and let $K_2(F)$ be the second $K$-group of $F$. It is known that
\[
K_2(F)\simeq H_2(\SL(F),\z).
\]
By a theorem of Matsumoto (see \cite[Theorem 11.1]{milnor1971}), we have an isomorphism
\[
K_2(F) \simeq (F^\times \otimes_\z  F^\times)/\lan a\otimes(1-a): a\in F\backslash \{0,1\}\ran.
\]
We denote the image of $\overline{a\otimes b}$, in $K_2(F)$, by $\{a,b\}$.

Let $A$ be a Euclidean domain with quotient field $F$. If $\ppp=\lan \pi \ran\in A$ is a non-zero prime ideal of $A$,
then $A_\ppp$ is a discrete valuation ring. Let $k(\ppp):=A_\ppp/\ppp A_\ppp$ be the residue field of $A_\ppp$.
Since $A$  is a Euclidean domain, $k(\ppp)\simeq A/\ppp$.
The ring $A_\ppp$ induces a discrete valuation 
\[
v_\ppp: F^\times \arr \z
\]
on $F$. This valuation defines the following {\it tame symbol} on $K_2(F)$:
\[
\tau_\ppp: K_2(F) \arr k(\ppp)^\times, \ \ \ \{a,b\}\mapsto (-1)^{v_p(a)v_p(b)}\displaystyle
\overline{\Bigg(\frac{b^{v_p(a)}}{a^{v_p(b)}}\Bigg)}
\]
(see \cite[p. 98]{milnor1971} or \cite[Lemma 6.3]{wei2013}). The next result is proved by Hutchinson.

\begin{prp}[Hutchinson]
Let $n>1$ be an integer and $p>2$ a prime such that $p\nmid n$. Then the map $\delta_p$ coincides with 
the composition
\[
\begin{array}{c}
H_2(\SL_2(\z[1/pn]),\z) \arr H_2(\SL(\z[1/pn]),\z) \arr H_2(\SL(\q),\z) \simeq K_2(\q) 
\overset{\tau_p}{\larr} \F_p^\times,
\end{array}
\]
where $\tau_p$ is the tame symbol on $K_2(\q)$ induced by the prime ideal $\lan p\ran$ of $\z$.
\end{prp}
\begin{proof}
See \cite[Proposition 5.6]{h2016}.
\end{proof}

For a prime $p$, let 
\[
\begin{array}{c}
r_p:=\rank\ H_2(\SL_2(\z[1/p]),\z).
\end{array}
\]
By Theorem \ref{AN} (of Adem-Naffah), we have
\[
\begin{array}{c}
r_p:= \begin{cases}
1 & \text{if $p=2,3$}\\
(p-7)/6 & \text{if $p\equiv 1 \pmod {12}$}\\
(p+1)/6 & \text{if $p\equiv 5 \pmod {12}$}\\
(p-1)/6 & \text{if $p\equiv 7 \pmod {12}$}\\
(p+7)/6 & \text{if $p\equiv 11 \!\!\!\pmod {12}$}
\end{cases}
\end{array}.
\]
Observe that $r_p$ is always odd. Clearly, this is true for $p=2,3$. If $p\equiv 1 \pmod {12}$, 
then $p-7\equiv -6 \pmod {12}$. From this, we have $p-7=12k-6$ and thus 
\[
r_p=(p-7)/6=2k-1
\]
is odd. A parallel argument proves the results for the remaining cases.

\begin{exa}

The primes $2$, $3$, $5$, $7$ and $13$ are the only primes for which the value of $r_p$ is $1$. Beyond 
these, for any odd integer $l>1$, there can be no more than four primes with $r_p$ equal to $l$. 
%So, if 
%$p_1<p_2<p_3<p_4$ are primes with $r_{p_1}=r_{p_2}=r_{p_3}=r_{p_4}$, then
%\begin{align*}
%& p_1 \equiv 11 \!\!\!\pmod {12},\\
%& p_2 \equiv 5  \pmod {12},\\
%& p_3 \equiv 7  \pmod {12},\\
%& p_4 \equiv 1  \pmod {12}.
%\end{align*}
%The value of $r_p$ for primes smaller than $300$ are as follows \textcolor{blue}{((Objection 13):I do not know what to do, I like the table but it is just my taste...)}:
%\textcolor{red}{I also do like the table, but an idea to address this point is just provide a reference, like Adem-Naffah paper...}
%\begin{align*}
%& r_2=1     && r_3=1       && r_5=1        && r_7=1      && r_{11}=3   && r_{13}=1  && r_{17}=3 \\
%& r_{19}=3  && r_{23}=5    && r_{29}=5     && r_{31}=5   && r_{37}=5   && r_{41}=7  && r_{43}=7  \\
%& r_{47}=9  && r_{53}=9    && r_{59}=11    && r_{61}=9   && r_{67}=11  && r_{71}=13 && r_{73}=11 \\
%& r_{79}=13 && r_{83}=15   && r_{89}=15    && r_{97}=15  && r_{101}=17 && r_{103}=17&& r_{107}=19\\
%&r_{109}=17 && r_{113}=19  && r_{127}= 21  && r_{131}=23 && r_{137}=23 && r_{139}=23&& r_{149}=25\\   
%& r_{151}=25&& r_{157}=25  && r_{163}=27   && r_{167}=29 && r_{173}=29 && r_{179}=31&& r_{181}=29 \\
%& r_{191}=33&& r_{193}=31  && r_{197}=33   && r_{199}=33 && r_{211}=35 && r_{223}=37&& r_{227}=39\\
%& r_{229}=37&& r_{239}=41  && r_{241}=39   && r_{251}=43 && r_{257}=43 && r_{263}=45&& r_{269}=45\\
%& r_{271}=45&& r_{277}=45  && r_{281}=47   && r_{283}=47 && r_{293}=49.&&           &&           
%\end{align*}
\end{exa}

%Now, we are ready to prove the following theorem, which is Theorem B of introduction.
%We are now prepared to prove Theorem B, as stated in the introduction.
The following theorem (Theorem B from the introduction) will now be demonstrated.

\begin{thm}\label{exact2}
Let $n=p_1 \cdots p_l$, $l>1$, where $p_i$'s are distinct primes such that $r_{p_1}\leq \dots \leq r_{p_l}$. 
When $r_{p_i}=r_{p_{i+1}}$, for some $i$, we assume that $p_i < p_{i+1}$. 
Then we have the exact sequence
\[
\begin{array}{c}
H_2(\SL_2(\z[1/p_1]),\z) \arr H_2(\SL_2(\z[1/n]),\z) \arr 
\bigoplus_{i=2}^l \F_{p_i}^\times \arr 1.
\end{array}
\]
\end{thm}
\begin{proof}
The proof is by induction on $l$. If $l=2$, then the claim follows from Proposition~\ref{exact1}.
%because $p_2\nmid p_1^2-1$. \textcolor{red}{It's clear for me, but it may be not so clear to 
%someone else. We are using the first item of Proposition 6.1, i.e. we are assuming that $p_2 > 3$. 
%But, for example, if $p_1 = 2$ and $p_2 = 3$, then $3 \mid 2^2 - 1 = 3$}. 
So, let $l\geq 3$. Thus, clearly 
\[
p_l\nmid \gcd\{m^2-1: m\mid p_1p_2\cdots p_{l-1}\}
\]
and so, by Proposition~\ref{exact1}, we have the exact sequence
\[
\begin{array}{c}
H_2(\SL_2(\z[{1}/{(p_1\cdots p_{l-1})}]),\z) \arr H_2(\SL_2(\z[1/n]),\z) \arr \F_{p_l}^\times \arr 1.
\end{array}
\]
Now, the claim follows by induction and applying the Snake lemma to the commutative diagram with exact rows
\[
\begin{tikzcd}
& H_2(\SL_2(\z[{1}/{(p_1\cdots p_{l-1})}]),\z)\ar[r]\ar[d, "{\delta_{p_2, \dots,p_{l-1}}}"] 
& H_2(\SL_2(\z[1/n]),\z)\ar[r]\ar[d, "{\delta_{p_2, \dots,p_{l}}}"] & 
\F_{p_l}^\times\ar[d, equal]\ar[r] &1 \\
1 \ar[r] & \bigoplus_{i=2}^{l-1} \F_{p_i}^\times\ar[r]&\bigoplus_{i=2}^l \F_{p_i}^\times\ar[r]&
\F_{p_l}^\times \ar[r] & 1.
\end{tikzcd}
\]
\end{proof}

\begin{prp}\label{rq=1-inj}
Let $m$ be a square-free integer that is divisible by at least one of the primes $2$, $3$, $5$, $7$ or $13$.
If $m$ divides a square-free integer $n$, then the natural map
\[
\begin{array}{c}
H_2(\SL_2(\z[1/m]),\z) \arr H_2(\SL_2(\z[1/n]),\z)
\end{array}
\]
is injective. Moreover, the left homology group is a direct summand of the right homology group 
if $m$ is equal to one of the primes $2$, $3$ or $7$.
\end{prp}
\begin{proof}
Firstly, we assume that $m$ is a prime. Thus, $m$ must be one of the primes $2$, $3$, $5$, $7$ or $13$. Let us
denote $m$ by $q$. In this case, we start our argument by showing that the natural map 
\[
\begin{array}{c}
H_2(\SL_2(\z[1/q]),\z) \arr H_2(\SL_2(\z[{1}/{6q}]),\z)
\end{array}
\]
is injective. If $q=2,3$, then $\z[{1}/{6q}]=\z[{1}/{6}]$. By Proposition \ref{exact1},
we have the exact sequences
\[
\begin{array}{c}
H_2(\SL_2(\z[{1}/{2}]),\z) \arr H_2(\SL_2(\z[{1}/{6}]),\z) \arr \F_3^\times \arr 1,
\end{array}
\]
\[
\begin{array}{c}
H_2(\SL_2(\z[1/3]),\z) \arr H_2(\SL_2(\z[1/6]),\z) \arr \F_2 \arr 0.
\end{array}
\]
By the calculation of Bui-Ellis in \cite[Table 1]{ae2014}, we know that 
\[
\begin{array}{c}
H_2(\SL_2(\z[1/6]),\z)\simeq \z \oplus \z/2.
\end{array}
\]
Since $H_2(\SL_2(\z[1/2]),\z)\simeq \z$ and $H_2(\SL_2(\z[1/3]),\z)\simeq \z$
(see Theorem \ref{AN}), the left maps in the above exact sequences are injective. In fact, we have 
\[
\begin{array}{c}
H_2(\SL_2(\z[1/6]),\z)\simeq H_2(\SL_2(\z[1/2]),\z) \oplus \F_3^\times,
\end{array}
\]
\[
\begin{array}{c}
H_2(\SL_2(\z[1/6]),\z)\simeq H_2(\SL_2(\z[1/3]),\z) \oplus \F_2.
\end{array}
\]
This proves the claim for $q=2,3$. So let $q\in \{5,7,13\}$. Observe that for these primes, by Theorem \ref{AN}, 
we have
\[
H_2(\SL_2(\z[1/q]),\z)\simeq \z \oplus \z/((q-1)/2).
\]
Moreover, by Theorem \ref{H26} (of Hutchinson),
\[
\begin{array}{c}
H_2(\SL_2(\z[1/6q]),\z)\simeq \z \oplus \F_{3}^\times\oplus \F_{q}^\times\simeq \z \oplus \z/2 \oplus \z/(q-1).
\end{array}
\]
More precisely, we have the split exact sequence
\[
\begin{array}{c}
0 \arr H_2(\SL_2(\z[1/2]),\z) \arr H_2(\SL_2(\z[1/6q]),\z) \arr 
\F_3^\times \oplus \F_{q}^\times \arr 1.
\end{array}
\]
On the other hand, by Proposition \ref{exact1}, we have the exact sequences
\[
\begin{array}{c}
H_2(\SL_2(\z[1/3q]),\z) \arr H_2(\SL_2(\z[1/6q]),\z) \arr \F_2 \arr 0,
\end{array}
\]
\[
\begin{array}{c}
H_2(\SL_2(\z[1/q]),\z) \arr H_2(\SL_2(\z[1/3q]),\z) \arr \F_{3}^\times\arr 0.
\end{array}
\]
Combining these two (as in Theorem \ref{exact2}), we obtain the exact sequence
\[
\begin{array}{c}
H_2(\SL_2(\z[1/q]),\z) \arr H_2(\SL_2(\z[1/6q]),\z) \arr \F_2\oplus\F_{3}^\times\arr 0.
\end{array}
\]
Now, using the above exact sequence of Hutchinson, the sequence
\[
\begin{array}{c}
0\arr H_2(\SL_2(\z[1/q]),\z)\arr H_2(\SL_2(\z[1/6q]),\z)\arr\F_2\oplus\F_{3}^\times\arr 0
\end{array}
\]
is exact and it splits only if $q=7$. Finally, the general claim (for the case $m=q$ a prime) follows from 
the commutative diagram
\[
\begin{tikzcd}
 H_2(\SL_2(\z[1/q]),\z) \ar[r]\ar[d, hook] & H_2(\SL_2(\z[1/n]),\z)\ar[d]  \\
 H_2(\SL_2(\z[1/6q]),\z)\ar[r, hook] & H_2(\SL_2(\z[1/6n]),\z)
\end{tikzcd}
\]
and the fact that $H_2(\SL_2(\z[{1}/{6q}]),\z)$ is a direct summand of $H_2(\SL_2(\z[1/6n]),\z)$
(this follows from Theorem \ref{H26}).

Now, let $q$ be the smallest prime among $2,3,5,7,13$ that divides $m$. Let $m=qp_1\cdots p_h$
and $n=mp_{h+1}\cdots p_l=qp_1\cdots p_l$. By Theorem \ref{exact2} and what we have just proved, we 
get the commutative diagram with exact rows
\[
\begin{tikzcd}
0 \ar[r] & H_2(\SL_2(\z[1/q]),\z)\ar[r]\ar[d, equal] & H_2(\SL_2(\z[1/m]),\z)\ar[r]\ar[d] & 
\bigoplus_{i=1}^h\F_{p_i}^\times\ar[d, hook]\ar[r] & 1 \\
0 \ar[r] & H_2(\SL_2(\z[1/q]),\z)\ar[r] & H_2(\SL_2(\z[1/n]),\z) \ar[r] &\bigoplus_{i=1}^l 
\F_{p_i}^\times \ar[r] & 1.
\end{tikzcd}
\]
The general injectivity claim of the theorem follows from the Snake lemma applied to the above diagram.
\end{proof}

%\textcolor{blue}{(Prof. Jaikin C1)(For the next result, remember that by Theorem \ref{AN}, 
%$H_2(\SL_2(\z[1/q]),\z)\simeq \z \oplus \z/((q-1)/2)$, for the primes $q=2,3,5,7,13$)} 

The next theorem, Theorem A from the introduction, represents a significant generalization of Hutchinson's 
Theorem~\ref{H26}.

\begin{thm}\label{rp=1}
Let $n$ be a square-free positive integer and let 
$\spcd(n,2730)$ be the smallest common prime divisor of $n$ and $2730=2\cdot 3 \cdot 5 \cdot 7 \cdot 13$.
\par {\rm (i)} If $\spcd(n,2730)=2$, i.e. $n$ is even, then 
\[
H_2(\SL_2(\z[1/n]),\z) \simeq \z\oplus  \bigoplus_{p\mid (n/2)} \F_{p}^\times.
\]
\par {\rm (ii)} If $\spcd(n,2730)=3$, then 
\[
H_2(\SL_2(\z[1/n]),\z) \simeq \z\oplus  \bigoplus_{p\mid (n/3)} \F_{p}^\times.
\]
\par {\rm (iii)} If $\spcd(n,2730)=5$, then 
\[
H_2(\SL_2(\z[1/n]),\z) \simeq \z \oplus
\begin{cases}
\z/2  \oplus \displaystyle\bigoplus_{p\mid (n/5)}\F_{p}^\times & \text{if $p\equiv 1\!\!\!\! \pmod {4}$ for all $p\mid (n/5)$},\\
\z/4  \oplus \z/((q-1)/2)  \oplus \!\!\!\!\displaystyle\bigoplus_{p\mid (n/5q)}\!\!\! \F_{p}^\times & \text{if $q\equiv 3 \!\!\!\! \pmod {4}$ for some $q\mid (n/5)$.}
\end{cases}
\]
\par {\rm (iv)} If $\spcd(n,2730)=7$, then 
\[
%\begin{array}{c}
H_2(\SL_2(\z[1/n]),\z) \simeq \z\oplus \z/3\oplus \bigoplus_{p\mid (n/7)} \F_{p}^\times.
%\end{array}
\]
\par {\rm (v)} If $\spcd(n,2730)=13$, then 
\[
\begin{array}{c}
H_2(\SL_2(\z[1/n]),\z)   
\end{array}
\!\!\simeq \z \oplus
\begin{cases}
\z/6  \oplus \!\!\!\!\displaystyle\bigoplus_{p\mid (n/13)}\F_{p}^\times & \text{if $p\equiv 1\!\!\!\! \pmod {4}$ for all $p\mid (n/13)$,}\\
\z/12  \oplus \z/((q-1)/2)  \oplus \!\!\!\!\displaystyle\bigoplus_{p\mid (n/13q)}\!\!\!\! \F_{p}^\times & \text{if $q\equiv 3 \!\!\!\! \pmod {4}$ for some $q\mid (n/13)$.}
\end{cases}
\]
\end{thm}
\begin{proof}
Let $n=p_1\cdots p_l$ be the prime decomposition of $n$ such that $r_{p_1}\leq \cdots \leq r_{p_l}$. 
Moreover, if $r_{p_i}=r_{p_{i+1}}$, for some $i$, we assume that $p_i < p_{i+1}$.
\par (i) Since $n$ is even, we put $p_1=2$. 
By Theorem \ref{exact2} and Proposition \ref{rq=1-inj}, we get the exact sequence
\begin{align}\label{sss}
\begin{array}{c}
0 \arr H_2(\SL_2(\z[1/2]),\z) \arr H_2(\SL_2(\z[1/n]),\z) 
\arr \bigoplus_{i=2}^l \F_{p_i}^\times \arr 1.
\end{array}
\end{align}
Again, by Proposition \ref{rq=1-inj}, the exact sequence (\ref{sss}) splits. This, together with the isomorphism 
$H_2(\SL_2(\z[1/2]),\z)\simeq \z$ (Theorem \ref{AN}), imply the first item.
\par (ii) Since $\spcd(n,2730)=3$, we have $p_1=3$. Now, a similar argument as in (i) proves
the second item.
\par (iv) If $\spcd(n,2730)=7$, then we have $p_1=7$. 
%(Note that, since $3\mid 7^2-1$, we assumed as hypothesis that $3\nmid n$.) 
%\textcolor{red}{(This is already in the hypothesis of statement)}. 
%If $r_{p_i}=r_{p_{i+1}}$, for some $i>1$, we assume that $p_i < p_{i+1}$. 
Again, by a similar argument as in (i), we can prove the fourth item.

\par (iii) Now, let $\spcd(n,2730)=5$. 
%(Since $3\mid 5^2-1$, similar to the above argument we assumed that $3\nmid n$.) 
Then $p_1=5$. If $p$ is an odd prime which divides $n$, let $p-1 = 2^{s_p}m_p$, where 
$s_p \geq 1$ and $m_p$ is odd. Note that $p\equiv 1\!\! \pmod {4}$ if and only if $s_p>1$
and  $p\equiv 3\!\! \pmod {4}$ if and only if $s_p=1$.

First, let us assume that $l=2$, i.e. $n=5p_2$. By Theorem \ref{exact2} and 
Proposition \ref{rq=1-inj}, we get the exact sequence
\[
\begin{array}{c}
0 \arr H_2(\SL_2(\z[1/5]),\z) \arr H_2(\SL_2(\z[1/5p_2]),\z) \arr \F_{p_2}^\times \arr 1.
\end{array}
\]
Consider the commutative diagram with exact rows
\begin{equation}\label{5p}
\begin{tikzcd}
0 \ar[r] & H_2(\SL_2(\z[1/5]),\z)\ar[r, "i_\ast"]\ar[d, hook] 
& H_2(\SL_2(\z[1/5p_2]),\z)\ar[r]\ar[d] & \F_{p_2}^\times\ar[d, equal]\ar[r] & 1 \\
0 \ar[r] & H_2(\SL_2(\z[1/30]),\z)\ar[r] & H_2(\SL_2(\z[1/30p_2,\z) \ar[r] 
&\F_{p_2}^\times \ar[r] & 1.
\end{tikzcd}
\end{equation}
By Theorem \ref{H26}, the lower exact sequence splits. Note that, by Theorem \ref{AN}, we have 
\[
\begin{array}{c}
H_2(\SL_2(\z[1/5]),\z)\simeq \z\oplus \z/2.
\end{array}
\]
It follows from this that the free part of $H_2(\SL_2(\z[1/5p_2]),\z)$ splits naturally. 
Observe that in the commutative diagram
\[
\begin{tikzcd}
\z\oplus \z/2 \ar[r] \ar[d, "\simeq"] & \z \oplus \z/2 \oplus \z/4\ar[d, "\simeq"]\\
H_2(\SL_2(\z[1/5]),\z) \ar[r] & H_2(\SL_2(\z[1/30]),\z)
\end{tikzcd}
\]
the upper map is given by $(a,\overline{b})\mapsto (a, 0, 2\overline{b})$. 

%\textcolor{red}{Correct also: change $m_2$ to $m_{p_2}$}.

Let $s_{p_2}>1$. Note that $\F_{p_2}^\times\simeq \z/(p_2-1)\simeq \z/2^{s_{p_2}}\oplus \z/m_{p_2}$.
%\textcolor{blue}{($\F_{p_2}^\times\simeq \z/(p_2-1)\simeq \z/2^{s_{p_2}}\oplus \z/\textcolor{red}{m_{p_2}}$)}.
If the upper exact sequence of the diagram (\ref{5p}) does not split, then 
\[
\begin{array}{c}
H_2(\SL_2(\z[1/5p_2]),\z)\simeq \z \oplus \z/2^{s_{p_2}+1}\oplus \z/m_{p_2}
%\ \ \textcolor{blue}{H_2(\SL_2(\z[1/5p_2]),\z)\simeq \z \oplus \z/2^{s_{p_2}+1}\oplus \z/\textcolor{red}{m_{p_2}}, }
\end{array}
\]
(see \cite[Theorem 3.4.3, 3.3.2]{wei1994}).
It follows from this that $H_2(\SL_2(\z[1/5p_2]),\z)$ has an element of order $s_{p_2}+1$. Moreover,
$H_2(\SL_2(\z[1/5p_2]),\z)$ injects into $H_2(\SL_2(\z[1/30p_2]),\z)$. But this later group 
does not has any element of order $s_{p_2}+1$. This is a contradiction and therefore the upper row of  
(\ref{5p}) must split. Therefore,
\[
\begin{array}{c}
H_2(\SL_2(\z[1/5p_2]),\z)\simeq \z \oplus \z/2 \oplus \F_{p_2}^\times.
\end{array}
\]

Now, let $s_{p_2}=1$. If the upper exact sequence of the diagram (\ref{5p}) splits, then
\[
\begin{array}{c}
H_2(\SL_2(\z[1/5p_2]),\z)\simeq \z \oplus \z/2 \oplus \F_{p_2}^\times \simeq 
\z \oplus \z/2 \oplus \z/2\oplus \z/m_{p_2}.
\end{array}
\]
But this contradicts the surjectivity of the map
\[
\begin{array}{c}
\delta_5: H_2(\SL_2(\z[1/5p_2]),\z) \two \F_5^\times\simeq \z/4
\end{array}
\]
(see Lemma \ref{H2-to-Fp}). Observe that $\delta_5$ maps the free part of 
$H_2(\SL_2(\z[1/5p_2]),\z)$ to zero. Therefore the first row of the diagram  
(\ref{5p}) does not split and hence,
\[
\begin{array}{c}
H_2(\SL_2(\z[1/5p_2]),\z) \simeq \z \oplus \z/4 \oplus \z/m_{p_2}\simeq \z\oplus \z/4 \oplus \z/((p_2-1)/2)
\end{array}
\]
(see Lemma \ref{H2-to-Fp}). This proves the claim of item (iii) for $l=2$. 

Now, let $l>2$ and consider the commutative diagram with exact rows
\[
\begin{tikzcd}
0\! \ar[r]\! & \!H_2(\SL_2(\z[1/5]),\z)\!\ar[r]\ar[d, equal] &\! H_2(\SL_2(\z[{1}/{(5p_2\cdots p_{l-1})}]),\z)
\!\ar[r]\ar[d] &\! \bigoplus_{i=2}^{l-1} \F_{p_i}^\times\ar[d, hook]\!\ar[r] &\! 1 \\
0\! \ar[r] & \! H_2(\SL_2(\z[1/5]),\z)\!\ar[r] &\! H_2(\SL_2(\z[1/n]),\z) \!\ar[r] & 
\!\bigoplus_{i=2}^l \F_{p_i}^\times \!\ar[r] & \!1.
\end{tikzcd}
\]
If $s_{p_j}>1$, for all $2\leq j\leq l$, then, by induction, the first row splits. Let $\eta$ be the composition
\[
\begin{array}{c}
\bigoplus_{i=2}^{l-1} \F_{p_i}^\times \arr H_2(\SL_2(\z[{1}/{(5p_2\cdots p_{l-1})}]),\z) \arr 
H_2(\SL_2(\z[1/n]),\z)
\end{array}
\]
and consider the exact sequence
\[
\begin{array}{c}
0 \arr H_2(\SL_2(\z[1/5]),\z)\arr H_2(\SL_2(\z[{1}/{(5p_2\cdots p_l)}]),\z)/\im(\eta) \arr
\F_{p_l}^\times \arr  1.
\end{array}
\]
Then, as in case $l=2$, one can show that the above exact sequence splits. Therefore
\[
\begin{array}{c}
H_2(\SL_2(\z[1/n]),\z) \simeq \z \oplus \z/2 \oplus \bigoplus_{i=2}^{l-1} \F_{p_i}^\times.
\end{array}
\]
Now, let  $s_{p_j}=1$, for some $j$. We may assume that $s_{p_2}=1$. Again, by induction 
on the first row of the above diagram, we have the decomposition 

\begin{align*}
H_2(\SL_2(\z[{1}/{(5p_2\cdots p_{l-1})}]),\z)  
\begin{array}{c}
%\simeq \z \oplus \z/2(p_2-1) \oplus \bigoplus_{i=3}^{l-1} \F_{p_i}^\times
%\end{array}
%&\\
%&
%\begin{array}{c}
\simeq \z \oplus \z/4\oplus \z/((p_2-1)/2) \oplus \bigoplus_{i=3}^{l-1} \F_{p_i}^\times.
\end{array}
%&
\end{align*}
By Theorem \ref{exact2} and Proposition \ref{rq=1-inj}, we have the exact sequence
\[
\begin{array}{c}
0 \arr H_2(\SL_2(\z[{1}/{(5p_2\cdots p_{l-1})}]),\z)\arr H_2(\SL_2(\z[1/n]),\z) \arr
\F_{p_l}^\times \arr  1.
\end{array}
\]
Let $\theta$ be the composition
\[
\begin{array}{c}
\bigoplus_{i=3}^{l-1} \F_{p_i}^\times \arr H_2(\SL_2(\z[{1}/{(5p_2\cdots p_{l-1})}]),\z) \arr 
H_2(\SL_2(\z[1/n]),\z).
\end{array}
\]
From this and the above exact sequence, we obtain the exact sequence
\[
\begin{array}{c}
0\arr\z\oplus \z/4 \oplus \z/((p_2-1)/2)\arr H_2(\SL_2(\z[1/n]),\z)/\im(\theta)\arr\F_{p_l}^\times\arr  1.
\end{array}
\]
Note that $H_2(\SL_2(\z[1/n]),\z)/\im(\theta)$ injects into
\[
\begin{array}{c}
H_2(\SL_2(\z[{1}/{(30p_2\cdots p_l)}]),\z)/\im(\theta)
\simeq \z \oplus \z/2 \oplus \z/4 \oplus \F_{p_2}^\times \oplus \F_{p_l}^\times.
\end{array}
\]
But this is only possible if the above exact sequence splits, i.e.
\[
\begin{array}{c}
H_2(\SL_2(\z[1/n]),\z) \simeq \z \oplus \z/4 \oplus \z/((p_2-1)/2) \oplus \bigoplus_{i=3}^{l} \F_{p_i}^\times.
\end{array}
\]
This completes the proof of the item (iii) of the theorem. 

\par (v) Let $\spcd(n,2730)=13$. 
%(Since $3,7\mid {13}^2-1$, we assumed as hypothesis that $3,7\nmid n$.) 
The proof of this part is similar to the proof of item~(iii), but is more involved.
Take $p_1=13$. First, let us assume that $l=2$, i.e. $n=13p_2$. By 
Theorem \ref{exact2} and Proposition \ref{rq=1-inj}, we get the exact sequence
\[
\begin{array}{c}
0 \arr H_2(\SL_2(\z[1/13]),\z) \arr H_2(\SL_2(\z[1/13p_2]),\z) \arr \F_{p_2}^\times \arr 1.
\end{array}
\]
Consider the commutative diagram with exact rows
\begin{equation}\label{13p}
\begin{tikzcd}
0 \ar[r] & H_2(\SL_2(\z[1/13]),\z)\ar[r, "i_\ast"]\ar[d, hook] 
& H_2(\SL_2(\z[1/13p_2]),\z)\ar[r]\ar[d] & \F_{p_2}^\times\ar[d, equal]\ar[r] & 1 \\
0 \ar[r] & H_2(\SL_2(\z[1/78]),\z)\ar[r] & H_2(\SL_2(\z[1/78p_2,\z) \ar[r] 
&\F_{p_2}^\times \ar[r] & 1.
\end{tikzcd}
\end{equation}
Note that $78=6 \cdot 13$.
By Theorem \ref{H26}, the lower exact sequence splits. Note that, by Theorem \ref{AN}, we have 
\[
\begin{array}{c}
H_2(\SL_2(\z[1/13]),\z)\simeq \z\oplus \z/6\simeq \z\oplus \z/2\oplus \z/3.
\end{array}
\]
Similar to the case $p_1=5$, the free part of $H_2(\SL_2(\z[1/13p_2]),\z)$ splits naturally
and in the commutative diagram
\[
\begin{tikzcd}
\z\oplus \z/6 \ar[r] \ar[d, "\simeq"] & \z \oplus \z/2 \oplus \z/12\ar[d, "\simeq"]\\
H_2(\SL_2(\z[1/13]),\z) \ar[r] & H_2(\SL_2(\z[1/78]),\z)
\end{tikzcd}
\]
the upper map is given by $(a,\overline{b})\mapsto (a, 0, 2\overline{b})$. Let 
\[
p_2-1=2^{s_{p_2}}m_{p_2}=2^{s_{p_2}}3^tm_{p_2}',
\]
where $m_{p_2}=3^tm_{p_2}'$, $t\geq 0$ and $3\nmid m_{p_2}'$. Then
\[
\F_{p_2}^\times\simeq \z/(p_2-1)\simeq \z/2^{s_{p_2}}\oplus \z/3^t\oplus\z/m_{p_2}'.
\]

Let $s_{p_2}>1$.
If the upper exact sequence of the diagram (\ref{13p}) does not split, then $H_2(\SL_2(\z[1/13p_2]),\z)$ 
has either a copy of $\z/2^{s_{p_2}+1}$ or a copy of $\z/3^{t+1}$, in case $t\geq 1$, as direct summand. But
$H_2(\SL_2(\z[1/13p_2]),\z)$ injects into $H_2(\SL_2(\z[1/78p_2]),\z)$ and this later group 
does not has any element of order $s_{p_2}+1$ or $t+1$, in case $t\geq 1$, since
\begin{align*}
H_2(\SL_2(\z[1/78p_2]),\z) &\simeq \z \oplus \F_3^\times \oplus \F_{13}^\times\oplus \F_{p_2}^\times\\
& \simeq \z \oplus \z/2 \oplus \z/12 \oplus \z/2^{s_{p_2}}\oplus \z/3^t\oplus\z/m_{p_2}'
\end{align*}
(see \cite[Theorem 3.4.3, Calculation 3.3.2, Exercise 3.4.1]{wei1994}). This is a contradiction and therefore the upper row of  
(\ref{13p}) must split. Therefore,
\[
\begin{array}{c}
H_2(\SL_2(\z[1/13p_2]),\z)\simeq \z \oplus \z/6 \oplus \F_{p_2}^\times
\end{array}.
\]

Now, let $s_{p_2}=1$. If the upper exact sequence of the diagram (\ref{13p}) splits, then
\[
\begin{array}{c}
H_2(\SL_2(\z[1/13p_2]),\z)\simeq \z \oplus \z/6 \oplus \F_{p_2}^\times \simeq 
\z \oplus \z/2 \oplus \z/2\oplus \z/3 \oplus \z/m_{p_2}.
\end{array}
\]
But this contradicts the surjectivity of the map
\[
\begin{array}{c}
\delta_{13}: H_2(\SL_2(\z[1/13p_2]),\z) \two \F_{13}^\times\simeq \z/4\oplus \z/3
\end{array}
\]
(see Lemma \ref{H2-to-Fp}). 
%Observe that $\delta_13$ maps the free part of $H_2(\SL_2(\z[1/13p_2]),\z)$ to zero. 
Therefore the first row of the diagram (\ref{13p}) does not split. If $t\geq 1$, we have seen that
$H_2(\SL_2(\z[1/13p_2]),\z)$ can not have a copy of $\z/3^{t+1}$. Thus, the only remaining case is 
the isomorphism
\[
\begin{array}{c}
H_2(\SL_2(\z[1/13p_2]),\z) \simeq \z \oplus \z/12 \oplus \z/m_{p_2}\simeq \z\oplus \z/12\oplus \z/((p_2-1)/2)
\end{array}
\]
(see \cite[Exercise 3.4.1]{wei1994}).
This proves the claim of item (v) for $l=2$. The case $l>2$ can be done as in the proof of (iii).
This completes the proof of  the theorem.
\end{proof}

%\textcolor{red}{(Remark X4) This connection with functor Ext was useful in determining the splitting 
%in the proof of the last theorem. Despite this, we can remove the remark...}
%\begin{rem}\label{extt}
%There is a one to one correspondence between the exact sequences of abelian groups of the form
%\[
%0\arr \z/2 \arr \Aa \arr \z/2m \arr 0,
%\]
%up to isomorphism, and the elements of the extension group $\ext(\z/2m, \z/2)$ (see 
%\cite[Theorem 3.4.3]{wei1994}). Since 
%\[
%\ext(\z/2m, \z/2)\simeq \z/2,
%\]
%(\cite[page 73, 3.3.2]{wei1994}), up to isomorphism, we only have the following exact sequences of the above form:
%\[
%0\arr \z/2 \overset{\alpha}{\larr} \z/2\oplus \z/2m \arr \z/2m \arr 0, \ \ \ \alpha(\bar{1})= (\bar{1},0),
%\]
%\[
%0\arr \z/2 \overset{\beta}{\larr} \z/4m \arr \z/2m \arr 0, \ \ \ \beta(\bar{1})= 2\bar{m}.
%\]
%Observe that of these exact sequences, the first one splits while the second one does not.
%\end{rem}

\begin{rem}
%By Theorem \ref{rp=1}, we have the isomorphism $H_2(\SL_2(\z[{1}/{46}]),\z)\simeq \z\oplus \z/22$.
%But, in \cite[Table 1]{ae2014}, we find the isomorphism
%\[
%\begin{array}{c}
%H_2(\SL_2(\z[{1}/{46}]),\z)\simeq \z/22.
%\end{array}
%\]
%This is a notation error in \cite{ae2014}, as we found the correct formulation 
%in \cite[Table 3.1, page~27]{bui2015}.
%
By Theorem \ref{rp=1}, we have $H_2(\SL_2(\z[{1}/{46}]),\z)\simeq \z\oplus \z/22$. However, \cite[Table 1]{ae2014} 
states the isomorphism $H_2(\SL_2(\z[{1}/{46}]),\z)\simeq\z/22$. This is a notational error in \cite{ae2014}; the 
correct isomorphism is confirmed in \cite[Table 3.1, p.~27]{bui2015}.
\end{rem}

%%%%%%%%%%%%%%%%%%%%%%%%%%%%%%%%%%%%%%%%%%%%%%%%%%%%%%%%%%%%%%%%%%%%%%%%%%%%%%%%%%
\section{On the rank of the second homology of \texorpdfstring{$\SL_2(\z[1/n])$}{Lg}}\label{sec7}
%%%%%%%%%%%%%%%%%%%%%%%%%%%%%%%%%%%%%%%%%%%%%%%%%%%%%%%%%%%%%%%%%%%%%%%%%%%%%%%%%%

For any $n>1$, let 
\[
\begin{array}{c}
r_n:=\rank\ H_2(\SL_2(\z[1/n]),\z).
\end{array}
\]
We have already seen the value of $r_p$ for any prime $p$. 
%Now, we can prove the following result.

\begin{prp}\label{rank=1}
Let $n>1$ be a square free integer. If one of the primes $2$, $3$, $5$, $7$ or $13$ divides $n$, then $r_{n}=1$.
\end{prp}
\begin{proof}
This follows from Theorem \ref{exact2}, Proposition \ref{rq=1-inj} and the fact that, for a prime
$p$, $r_p=1$ if and only if $p=2,3,5,7 \text{ or } 13$. But this also follows from the next proposition, 
which can be proved much easier.
\end{proof}

\begin{prp}\label{rank}
For any square free  integer $n >1$, 
\[
1\leq r_n\leq \min \{r_p: p \ {\rm prime},\ p\mid n\}.
\]
\end{prp}
\begin{proof}
It follows from Theorem \ref{exact2} that 
%\[
$r_n\leq \min \{r_p: p \ {\rm prime},\ p\mid n\}$.
%\]
But here we give a much easier proof of this fact. We may assume that $n=p_1\cdots p_l$, where 
$p_i$'s are distinct primes with $r_{p_1} \leq r_{p_2} \leq \cdots \leq r_{p_l}$.
Let $m=p_1\cdots p_{l-1}$. From the diagram (\ref{ext}), we obtain the morphism of 
Lyndon/Hochschild-Serre spectral sequences
\[
\begin{tikzcd}
E_{r,s}^2=H_r(\B(\F_p),H_s(\Gamma(m, p_l),\q)) \ar[r,Rightarrow]\ar[d]& H_{r+s}(\Gamma_0(m, p_l),\q)\ar[d]\\
\EE_{r,s}^2=H_r(\SL_2(\F_p), H_s(\Gamma(m, p_l),\q)) \ar[r, Rightarrow]& H_{r+s}(\SL_2(\z[1/m]),\q).
\end{tikzcd}
\]
Since $E_{r,s}^2=0$ and $\EE_{r,s}^2=0$, for $r>0$ (see \cite[Corollary 10.2, \S10, Chap III]{brown1994}), 
and $E_{0,s}^2 \arr \EE_{0,s}^2$ is surjective, we see that, for any $k\geq 0$, the map
\begin{equation}\label{SURJ}
\begin{array}{c}
H_k(\Gamma_0(m, p_l),\q)\arr  H_k(\SL_2(\z[1/m]),\q)
\end{array}
\end{equation}
is surjective. Since $H_1(\Gamma_0(m, p_l),\z)$ is a finite group, $H_1(\Gamma_0(m, p_l),\q)=0$. Now, by the 
Mayer-Vietoris exact sequence
\begin{align*}
\begin{array}{c}
H_2(\Gamma_0(m, p_l),\q) \!\arr\! H_2(\SL_2(\z[1/m]),\q)\! \oplus\! H_2(\SL_2(\z[1/m]),\q) 
\!\arr\!  H_2(\SL_2(\z[1/n]),\q) \!\arr\! 0
\end{array}
\end{align*}
(see Section \ref{sec3}) and the surjective map (\ref{SURJ}), we see that 
\begin{equation}\label{SURJ2}
\begin{array}{c}
H_2(\SL_2(\z[1/m]),\q) \arr  H_2(\SL_2(\z[1/n]),\q)
\end{array}
\end{equation}
is surjective. Thus, by induction on $l$, we have
\[
r_n\leq r_{p_1}=\min \{r_p: p \ {\rm prime},\ p\mid n\}.
\]
Now, let $p\mid n$ and consider the commutative diagram
\[
\begin{tikzcd}
H_2(\SL_2(\z[1/p]),\q) \ar[r] \ar[d] & H_2(\SL_2(\z[1/n]),\q)\ar[d]\\
H_2(\SL_2(\z[1/6p]),\q) \ar[r] & H_2(\SL_2(\z[1/6n]),\q).
\end{tikzcd}
\]
By Theorem \ref{H26}, $r_{6n}=1$. Thus, the lower horizontal map in the above diagram is bijective, i.e.
\[
\begin{array}{c}
H_2(\SL_2(\z[1/6p]),\q) \simeq \q \simeq  H_2(\SL_2(\z[1/6n]),\q).
\end{array}
\]
Moreover, by (\ref{SURJ2}), the vertical maps in the above diagram are surjective. It follows 
from these that the upper horizontal map is not trivial. Therefore, $r_n \geq 1$. 
\end{proof}

We strongly suspect that the value of $r_n$ in the above theorem is equal to its upper bound. 
At this time, we lack a definitive proof (or disproof) of this claim. However, building upon the 
results of this paper, we propose the following conjecture over the group structure of
$H_2(\SL_2(\z[1/n]),\z)$ when $n$ is an integer not divisible by any of the primes 
$2$, $3$, $5$, $7$ or $13$.

\begin{conj}\label{g-conj}
Let $n=p_1 \cdots p_l$, $l > 1$, be the prime decomposition of $n$ such that 
$1 < r_{p_1}\leq \dots \leq r_{p_l}$. If $r_{p_j}=r_{p_{j+1}}$, for some $j$,
we assume that $p_j < p_{j+1}$. 
\par {\rm (i)} If $p_1\equiv 11 \pmod {12}$, then 
\[
\begin{array}{c}
H_2(\SL_2(\z[1/n]),\z) \simeq \z^{r_{p_1}}\oplus  \bigoplus_{j=2}^l \F_{p_j}^\times.
\end{array}
\]
\par {\rm (ii)} If $p_1\equiv 5 \pmod {12}$, then 
\[
\begin{array}{c}
H_2(\SL_2(\z[1/n]),\z) \!\simeq \!
\end{array}
\z^{r_{p_1}} \oplus
\begin{cases}
%\begin{array}{c}
 \z/2  \oplus \bigoplus_{j=2}^l \F_{p_j}^\times 
%\end{array}
 & \text{if $p_j\equiv 1\!\!\!\! \pmod {4}$ for all $2\leq j\leq l$},\\
 \\
\z/4\!\oplus\! \z/((q_i-1)/2)\!  \oplus\! \bigoplus_{\underset{j\neq i}{j=2}}^l 
\F_{p_j}^\times & \text{if $p_i\equiv 3\!\!\!\! \pmod {4}$ for some $2\leq i\leq l$.}
\end{cases}
\]
\par {\rm (iii)} If $p_1\equiv 7 \pmod {12}$, then 
\[
\begin{array}{c}
H_2(\SL_2(\z[1/n]),\z) \simeq \z^{r_{p_1}}\oplus \z/3\oplus 
\bigoplus_{j=2}^l \F_{p_j}^\times.
\end{array}
\]
\par {\rm (iv)} If $p_1\equiv 1 \pmod {12}$, then 
\[
\begin{array}{c}
H_2(\SL_2(\z[1/n]),\z)\! \simeq \!
\end{array}
\!\z^{r_{p_1}} \oplus
\begin{cases}
\z/6  \oplus \bigoplus_{j=2}^l \F_{p_j}^\times & \text{if $p_j\equiv 1\!\!\! \pmod {4}$ for all $2\leq j\leq l$},\\
\\
\z/12\! \oplus\! \z/((q_i-1)/2) \! \oplus\! \bigoplus_{\underset{j\neq i}{j=2}}^l \F_{p_j}^\times 
& \text{if $p_i\equiv 3\!\!\! \pmod {4}$ for some $2\leq i\leq l$.}
\end{cases}
\]
\end{conj}

%%%%%%%%%%%%%%%%%%%%%%%%%%%%%%%%%%%%%%%%%%%%%%%%%%%%%%%%

\end{document}